\documentclass[a4paper,11pt]{amsart}

\usepackage{amsmath, amsthm, amsfonts, amssymb}
\usepackage[bookmarks=false]{hyperref}
\usepackage{enumerate}

\usepackage{color}

\numberwithin{equation}{section}

\newtheorem{letterthm}{Theorem}

\newtheorem{thm}{Theorem}[section]
\newtheorem{lem}[thm]{Lemma}

\newtheorem{cor}[thm]{Corollary}
\newtheorem{prop}[thm]{Proposition}

\theoremstyle{definition}

\newtheorem{example}[thm]{Example}

\newtheorem{df}[thm]{Definition}

\newtheorem*{newclaim}{Claim}

\newcommand{\R}{\mathbf{R}}
\newcommand{\C}{\mathbf{C}}
\newcommand{\Z}{\mathbf{Z}}
\newcommand{\F}{\mathbf{F}}
\newcommand{\T}{\mathbf{T}}

\newcommand{\N}{\mathbf{N}}
\newcommand{\B}{\mathbf{B}}

\newcommand{\cH}{\mathcal{H}}

\newcommand{\cK}{\mathcal{K}}

\newcommand{\Ad}{\operatorname{Ad}}
\newcommand{\id}{\text{\rm id}}

\newcommand{\Inn}{\operatorname{Inn}}
\newcommand{\Aut}{\operatorname{Aut}}
\newcommand{\Out}{\operatorname{Out}}
\newcommand{\rL}{\mathord{\text{\rm L}}}

\newcommand{\conv}{\overline{\mathord{\text{\rm conv}}} \,}

\newcommand{\rE}{\mathord{\text{\rm E}}}

\newcommand{\ovt}{\mathbin{\overline{\otimes}}}

\newcommand{\Mod}{\mathord{\text{\rm Mod}}}

\newcommand{\op}{^{\mathrm{op}}}
\newcommand{\rd}{\: \mathrm{d}}
\newcommand{\ri}{\mathrm{i}}

\newcommand{\II}{{\rm II}}
\newcommand{\III}{{\rm III}}

\begin{document}

\title[Tensor product decompositions and rigidity of full factors]{Tensor product decompositions\\ and rigidity of full factors}

\begin{abstract}
We obtain several rigidity results regarding tensor product decompositions of factors. First, we show that any full factor with separable predual has at most countably many tensor product decompositions up to stable unitary conjugacy. We use this to show that the class of separable full factors with countable fundamental group is stable under tensor products. Next, we obtain new primeness and unique prime factorization results for crossed products comming from compact actions of higher rank lattices (e.g.\ $\mathrm{SL}(n,\Z), \: n \geq 3$) and noncommutative Bernoulli shifts with arbitrary base (not necessarily amenable). Finally, we provide examples of full factors without any prime factorization.
\end{abstract}

\author{Yusuke Isono}
\email{isono@kurims.kyoto-u.ac.jp}
\thanks{YI is supported by JSPS KAKENHI Grant Number JP17K14201.}
\address{RIMS, Kyoto University, 606-8502 Japan}

\author{Amine Marrakchi}
\email{amine.marrakchi@math.u-psud.fr}
\thanks{AM is a JSPS International Research Fellow (PE18760)}
\address{RIMS, Kyoto University, 606-8502 Japan}



\maketitle


\section{Introduction}
A central theme in the theory of von Neumann algebras is to determine all possible tensor product decompositions of a given factor $M$. More precisely, we will say that a subfactor $P \subset M$ is a \emph{tensor factor} of $M$ if $M=P \ovt P^c$ where $P^c=P' \cap M$. We will denote by $\mathrm{TF}(M)$ the set of all tensor factors of $M$. The set $\mathrm{TF}(M)$ contains all type $\mathrm{I}$ subfactors of $M$. Moreover, if $P \in \mathrm{TF}(M)$, then $uPu^* \in \mathrm{TF}(M)$ for every unitary $u \in \mathcal{U}(M)$. In order to eliminate both of these trivialities, one introduces the following equivalence relation: two tensor factors $P,Q \in \mathrm{TF}(M)$ are called \emph{stably unitarily conjugate}, written $P \sim Q$, if there exists type $\mathrm{I}_\infty$ factors $F_1,F_2$ and a unitary $u \in \mathcal{U}(M \ovt F_1 \ovt F_2)$ such that $u(P \ovt F_1)u^*=Q \ovt F_1$. One then wants to study the quotient space $\mathrm{TF}(M)/{\sim}$.

In many cases, one can give a complete description of $\mathrm{TF}(M)/{\sim}$. Indeed, a celebrated result of Ozawa \cite{Oz03} says that for every ICC hyperbolic group $\Gamma$, the $\II_1$ factor $M=L(\Gamma)$ is \emph{prime}. This means that for every tensor factor $P \in \mathrm{TF}(M)$, we have that either $P$ or $P^c$ is of type $\mathrm{I}$, or equivalently that $\mathrm{TF}(M)/{\sim} = \{ [\C], [M] \}$. More generally, we say that a factor $M$ satisfies the \emph{Unique Prime Factorization (UPF) property} if there exists prime factors $P_1,\dots, P_n \in \mathrm{TF}(M)$ with $M=P_1 \ovt \cdots \ovt P_n$ such that for every $Q \in \mathrm{TF}(M)$, there exists a subset $\{i_1,\dots, i_m\} \subset \{1,\dots, n\}$ such that $Q \sim P_{i_1} \ovt \cdots \ovt P_{i_m}$. In \cite{OP03}, Ozawa and Popa showed that if $\Gamma_1,\dots, \Gamma_n$ are ICC hyperbolic groups, then the factor $M=L(\Gamma_1 \times \dots \times \Gamma_n)$ has the UPF property. These seminal results were later generalized to larger and larger classes of factors by using Popa's deformation/rigidity theory and Ozawa's $C^*$-algebraic techniques \cite{Po06b, Pe06, Sa09, CSU11, SW11, Is14, CKP14, HI15, Ho15, Is16, DHI17}.

The main goal of this paper is to provide new rigidity and classification results for tensor product decompositions by combining the following two approaches:

\subsubsection*{Rigidity of full factors.} A factor is called \emph{full} when it has no nontrivial central sequences \cite{Co74}. Fullness is a very weak rigidity property when compared to Kazhdan's property (T) for example. In this paper, we use the following key bimodule characterization of fullness due to Ozawa \cite{BMO19} (and based on \cite{Co75,Ma18}): a factor $M$ is full if and only if for every $M$-$M$-bimodule $\mathcal{H}$ such that $\rL^2(M) \prec \cH$ and $\cH \prec \rL^2(M)$, we have $\rL^2(M) \subset \cH$. Note that if we remove the condition $\cH \prec \rL^2(M)$, this becomes precisely the definition of property (T). Therefore, in some specific situations, in particular for tensor product decompositions, full factors can behave in a very rigid way, as if they had property (T). See for instance Lemma 5.2 which shows that ``relative amenability" can be automatically improved to ``corner embedability" for full tensor factors. This can be seen as an instance of the \emph{spectral gap rigidity} phenomenon discovered in \cite{Po06b}.

\subsubsection*{Flip automorphisms.} Let $M$ be a factor. To every $P \in \mathrm{TF}(M)$ one can associate an automorphism $\sigma_P \in \Aut(M \ovt M)$ which flips the two copies of $P$ in $M\ovt M=P \ovt P^c \ovt P \ovt P^c$ and fixes the two copies of $P^c$. The key point is that it is in general much easier to study the flip automorphism $\sigma_P$ then to study directly the mysterious tensor factor $P$. Any information obtained on $\sigma_P$ can then be used to locate $P$ inside $M$ (observe in particular that $P \sim Q$ if and only if $\sigma_P \circ \sigma_Q$ is an \emph{inner} automorphism). As we will see, this trick combines very well with $\mathrm{W}^*$-rigidity results, since they generally give a good understanding of the automorphism group $\Aut(M \ovt M)$ in terms of the building data of $M$. This approach can be used to obtain new primeness or UPF results which do not rely on any kind of negative curvature or rank 1 assumption, but, on the other hand, cannot be used to obtain solidity or relative solidity results.

\bigskip
 
Let us now state our main theorems. We start with a very general rigidity result based on a \emph{separability argument} (see \cite[Section 4]{Po06c} for a survey). Unlike the separability arguments used in \cite{Co80} \cite{Po86}, \cite{Oz02}, \cite{Hj02}, \cite{GP03}, the rigidity in our case comes from fullness instead of property (T). Here $\mathrm{TF}_{\mathrm{full}}(M) \subset \mathrm{TF}(M)$ denotes the set of all full tensor factors of $M$. Note that $\mathrm{TF}_{\mathrm{full}}(M) = \mathrm{TF}(M)$ when $M$ itself is full.

\begin{letterthm} \label{main topological}
Let $M$ be a factor with separable predual. Then $\mathrm{TF}_{\mathrm{full}}(M)/{\sim}$ is countable. Consequently, if $\Omega$ (resp.\ $\Omega_{\mathrm{full}}$) denotes the set of all stable isomorphism classes of (resp.\ full) factors with separable predual, then the natural map
\begin{gather*} \Omega \times \Omega_{\mathrm{full}} \rightarrow \Omega \\
([P],[Q]) \mapsto [P \ovt Q]
\end{gather*}
is countable-to-one.
\end{letterthm}
The fullness assumption in Theorem \ref{main topological} is essential since $\mathrm{TF}(M)/{\sim}$ is uncountable whenever $M$ is an infinite tensor product of $\II_1$ factors (i.e.\ a \emph{McDuff} factor). Note that if a factor $M$ satisfies the UPF property, then $\mathrm{TF}(M)/{\sim}$ is actually finite. In view of Theorem \ref{main topological} and of all the known UPF results in the litterature, one might wonder if there exists any full factor $M$ which does not satisfy the UPF property. We answer this question affirmatively in the last section of this paper by providing the first examples of full factors which do not admit any prime factorization. For these examples, $\mathrm{TF}(M)/{\sim}$ is infinite but can still be completely described.

In our next main result, we give an application of this rigidity phenomenon to fundamental groups. Let $M$ be a $\II_\infty$ factor. Then every $\theta \in \Aut(M)$ scales the trace of $M$ by some scalar $\Mod(\theta) \in \R^*_+$ and the map $\Mod : \Aut(M) \rightarrow \R^*_+$ is a continuous group homomorphism. Its image is called the \emph{fundamental group} of $M$ and denoted by $\mathcal{F}(M)$. The fundamental group $\mathcal{F}(M)$ is also defined when $M$ is a $\II_1$ factor by $\mathcal{F}(M)=\mathcal{F}(M^\infty)$ where $M^\infty=M \ovt \B(\ell^2)$. The invariant $\mathcal{F}(M)$ is very hard to compute in general. In fact, for a long time, the only known computation, due to Murray and von Neumann, was $\mathcal{F}(M)=\R^*_+$ where $M$ is the hyperfinite $\II_1$ factor (or more generally a McDuff factor). The first breakthrough is the rigidity result of Connes \cite{Co80} which shows that $\mathcal{F}(M)$ is \emph{countable} for $M=L(\Gamma)$ where $\Gamma$ is a countable ICC group with Kazhdan's property (T). Voiculescu and R$\rm \breve{a}$dulescu then proved $\mathcal{F}(L\F_\infty)=\R^*_+$ \cite{Vo89,Ra91} by using the free entropy machinery. Since $L\F_\infty$ is full, this example shows in particular that fullness does not imply countability of the fundamental group. Later on, spectacular progress in the study of fundamental groups has been accomplished thanks to Popa's deformation/rigidity theory \cite{Po01}, \cite{Po03}, \cite{PV10}, \cite{PV11}. In particular, in \cite{PV11}, Popa and Vaes settled a longstanding question by giving the first example of a $\II_\infty$ factor $M$ with $\mathcal{F}(M)=\R^*_+$ but such that $M$ does \emph{not} admit a trace scaling action, i.e.\ a continuous action $\theta : \R^*_+ \curvearrowright M$ such that $\Mod(\theta_\lambda)=\lambda$ for all $\lambda \in \R^*_+$. Moreover, they gave an example of two factors $M$ and $N$ such that $\mathcal{F}(M \ovt N)=\R^*_+$ but $\mathcal{F}(M) \neq \R^*_+$ and $\mathcal{F}(N) \neq \R^*_+$. This should be compared with item $(\rm ii)$ below.

\begin{letterthm} \label{trace scaling}
Let $M$ and $N$ be two $\II_\infty$ factors with separable predual and suppose that one of them is full. Then the following holds:
\begin{itemize}
\item [$(\rm i)$] The quotient group $\mathcal{F}(M \ovt N)/\mathcal{F}(M)\mathcal{F}(N)$ is countable. In particular, if $\mathcal{F}(M)$ and $\mathcal{F}(N)$ are countable, then $\mathcal{F}(M \ovt N)$ is also countable.
\item [$(\rm ii)$] $M \ovt N$ admits a trace scaling action if and only if $M$ or $N$ admits a trace scaling action.
\end{itemize}
\end{letterthm}
We point out that in many concrete cases, one can actually show that $\mathcal{F}(M \ovt N)=\mathcal{F}(M)\mathcal{F}(N)$. See \cite{Is16} for recent results regarding this question. Nevertheless, we believe that Theorem \ref{trace scaling} is optimal and that the equality $\mathcal{F}(M \ovt N) = \mathcal{F}(M)\mathcal{F}(N)$ does not hold in general, even though we do not know any counter-example.

We now move to more concrete applications. The first one is a UPF result for crossed products comming from noncommutative Bernoulli shifts. Here, by a \emph{noncommutative Bernoulli shift} we always mean an action of the form $\Gamma \curvearrowright (B_0,\varphi_0)^{\ovt \Gamma}$ where $B_0$ is a non-trivial von Neumann algebra with separable predual, $\varphi_0$ is a faithful normal state on $B_0$ and $\Gamma$ is a countable group acting by shifting the tensor components. It is known that if $\Gamma$ is non-amenable and $B_0$ is amenable, then the crossed product $(B_0,\varphi_0)^{\ovt \Gamma}$ is prime \cite{Po06b, SW11, Ma16}. By exploiting the fullness of $(B_0,\varphi_0)^{\ovt \Gamma} \rtimes \Gamma$ (see \cite{VV14}), we are able to remove the amenability assumption on $B_0$.

\begin{letterthm} \label{bernoulli strongly prime}
Let $\Gamma \curvearrowright (B_0,\varphi_0)^{\ovt \Gamma}$ be any noncommutative Bernoulli shift. Assume that $\Gamma$ is non-amenable. Then $M=(B_0,\varphi_0)^{\ovt \Gamma} \rtimes \Gamma$ is prime.

Moreover, for any full factor $N$ with separable predual and any tensor product decomposition $M \ovt N=P \ovt Q$, we have $M \prec_{M \ovt N} P$ or $M \prec_{M \ovt N} Q$.
\end{letterthm}
The second part of the theorem shows in particular that if $N$ is a full factor which has the UPF property, then so does $M \ovt N$. The technique used in the proof of Theorem \ref{bernoulli strongly prime} also allows us to prove the following rigidity result which generalizes \cite[Corollary 0.2]{Io06} where the base algebras $A_0$ and $B_0$ are assumed to be weakly rigid $\II_1$ factors (e.g.\ $\II_1$ factors with property (T)).

\begin{letterthm} \label{bernoulli remembers groups}
Let $G \curvearrowright (A_0,\psi_0)^{\ovt G}$ and $H \curvearrowright (B_0,\varphi_0)^{\ovt H}$ be two noncommutative Bernoulli shifts. Assume that $A_0$ and $B_0$ are diffuse full factors. If $(A_0,\psi_0)^{\ovt G} \rtimes G \cong (B_0,\varphi_0)^{\ovt H} \rtimes H$, then there exists two finite normal subgroups $G_0 \lhd G$ and $H_0 \lhd H$ such that $G/G_0 \cong H/H_0$.
\end{letterthm}

  In our next theorem, we present a new UPF result which shows how the flip automorphism approach can be used to study tensor factors. This result should really be considered as an application of a  recent $\mathrm{W}^*$-rigidity result of Boutonnet-Ioana-Peterson \cite{BIP18} for compact actions of higher rank lattices. Here, by \emph{higher rank} lattice we mean a lattice $\Gamma$ in a connected semi-simple Lie group $G$ with finite center such that every simple quotient of $G$ has real rank $\geq 2$. A basic example is given by $\Gamma=\mathrm{SL}(n,\Z)$ for $n \geq 3$. We also recall that an ergodic pmp action $\Gamma \curvearrowright (X,\mu)$ is \emph{compact} if the closure of the image of $\Gamma$ in $\Aut(X,\mu)$ is compact. These are precisely the actions of the form $\Gamma \curvearrowright K/L$ where $K$ is a compact group, $L$ is a closed subgroup of $K$ and $\Gamma < K$ is a dense subgroup which acts by left translations.

\begin{letterthm} \label{irreducible lattice thm}
Let $\Gamma$ be an irreducible higher rank lattice. Let $\Gamma \curvearrowright (X,\mu)$ be a compact free ergodic pmp action. Then the crossed product $M=\rL^\infty(X) \rtimes \Gamma$ is prime. Moreover, for any finite family of factors $M_1, \dots, M_n$ of that form, the tensor product $M_1 \ovt \cdots \ovt M_n$ has the Unique Prime Factorization property.
\end{letterthm}
We point out that the question of whether the group von Neumann algebras of irreducible higher rank lattices are prime is a well-known and notoriously difficult open problem. We mention, however, the remarkable result of \cite{DHI17} which  shows that $L(\Gamma)$ is prime whenever $\Gamma$ is an irreducible lattice in a direct product of \emph{rank one} simple Lie groups.

For our last application, we consider factors of the form $M=R \rtimes \Gamma$ where $\Gamma \curvearrowright R$ is a compact minimal action of an ICC higher rank lattice $\Gamma$. Recall that an action $\Gamma \curvearrowright R$  is \emph{minimal} if it is faithful and $(R^\Gamma)' \cap R=\C$. Since every compact group admits one and only one minimal action on the hyperfinite $\II_1$ factor \cite{MT06}, a compact minimal action $\Gamma \curvearrowright R$ is uniquely determined, up to conjugacy, by the pair $\Gamma < K$ where the compact group $K$ is the closure of $\Gamma$ in $\Aut(R)$. In this context, we show that all tensor factors of $M$ are McDuff so that one cannot classify them up to stable \emph{unitary} conjugacy. However, we prove a unique semi-prime factorization result up to conjugacy by an \emph{automorphism}. Recall that a factor $M$ is \emph{semi-prime} if it is nonamenable and for every tensor product decomposition $M=P \ovt Q$, either $P$ or $Q$ is amenable.  
 
\begin{letterthm} \label{compact minimal rigidity}
Let $\Gamma$ be an ICC higher rank lattice. Let $\Gamma \curvearrowright R$ be a compact minimal action on the hyperfinite $\II_1$ factor and put $M=R \rtimes \Gamma$. Then the following holds:
\begin{itemize}
\item [$(\rm i)$] Every tensor factor of $M$ is either of type $\mathrm{I}$ or McDuff.
\item [$(\rm ii)$] $M$ admits a tensor product decomposition $$M=(R_1 \rtimes \Gamma_1) \ovt \cdots \ovt (R_n \rtimes \Gamma_n)$$ where $\Gamma=\Gamma_1 \times \cdots \times \Gamma_n$ and $R=R_1 \ovt \cdots \ovt R_n$ such that each $M_i=R_i \rtimes \Gamma_i$ is semi-prime.
\item [$(\rm iii)$] For every tensor product decomposition $M=P \ovt Q$ with $P$ and $Q$ nonamenable, there exists a partition $I \sqcup J=\{1, \dots, n \}$ and an automorphism $\theta$ of $M$ such that $\theta(P)=\ovt_{i \in I} M_i$ and $\theta(Q)=\ovt_{j \in J} M_j$, up to equivalence in $\mathrm{TF}(M)$.
\end{itemize}
 In particular, $M$ admits a unique semi-prime factorization up to conjugacy by an automorphism.
\end{letterthm}

\subsection*{Acknowledgments} The authors are very grateful to Adrian Ioana for allowing us to include the proof of Theorem \ref{full stable gap} in our paper and for his valuable comments on an earlier draft of this paper. We also thank Narutaka Ozawa for a helpful discussion regarding Proposition \ref{meier group}.

\tableofcontents

\section{Preliminaries}
\subsection*{Ultraproduct von Neumann algebras}
Let $M$ be any von Neumann algebra. Let $I$ be any nonempty directed set and $\omega$ any cofinal ultrafilter on $I$. Define
\begin{align*}
	\ell^\infty(I,M) &= \{ (x_i)_{i\in I} \mid x_i\in M, \ \sup_i \|x_i\|_\infty <\infty \} \\
	\mathcal I_{\omega} &= \left\{ (x_i)_{i} \in \ell^\infty(I, M) \mid x_i \to 0 \text{ *-strongly as } i \to \omega \right\} \\
	\mathcal M^{\omega} &= \left \{ x \in \ell^\infty(I, M) \mid  x \mathcal I_{\omega} \subset \mathcal I_{\omega} \text{ and } \mathcal I_{\omega}x \subset \mathcal I_{\omega}\right\}.
\end{align*}
The quotient C$^*$-algebra $M^\omega := \mathcal{M}^\omega/ \mathcal{I}_\omega$ is in fact a von Neumann algebra, and we call it the \textit{ultraproduct von Neumann algebra} \cite{Oc85}. For more on ultraproduct von Neumann algebras, we refer the reader to \cite{Oc85,AH12}.

\subsection*{Topological groups associated to a von Neumann algebra}
Let $M$ be any von Neumann algebra and let $\mathcal{U}(M)$ be its unitary group. The restrictions of the weak topology, the strong topology and the $*$-strong topology all coincide on $\mathcal{U}(M)$. Equipped with this topology, $\mathcal{U}(M)$ is a complete topological group (which is Polish if $M_*$ is separable). Let $\Aut(M)$ be the group of all automorphisms of $M$. We equip it with the topology of pointwise norm convergence on $M_*$, which means that a net $(\alpha_i)_i $ in $\Aut(M)$ converges to $\alpha \in \Aut(M)$ if and only if $\| \phi\circ \alpha_i - \phi\circ \alpha \| \to 0$ for any $\phi \in M_*$. With this topology, $\Aut(M)$ is a complete topological group and it is Polish when $M_*$ is separable. There is continuous homomorphism
$$\Ad : \mathcal{U}(M) \ni u \mapsto \Ad(u) \in \Aut(M)$$
where $\Ad(u)(x)=uxu^*$ for all $x \in M$. We denote by $\Inn(M)\subset \Aut(M)$ the image of $\Ad$, i.e.\ the set of all inner automorphisms. Since $\Inn(M)$ is a normal subgroup in $\Aut(M)$, we can form the quotient group $\Out(M):=\Aut(M)/\Inn(M)$ which we call the \textit{outer automorphism group of $M$} and we equip it with the quotient topology (which is not necessarily Hausdorf).

For any von Neumann algebras $M$ and $N$, we have a natural continuous homomorphism 
$$\Aut(M) \times \Aut(N) \ni (\alpha,\beta) \mapsto \alpha \otimes \beta \in \Aut(M \ovt N)$$
which also induces a continuous injective homomorphism $$\Out(M) \times \Out(N) \rightarrow \Out(M \ovt N).$$ 

\subsection*{Full factors}
Following \cite{Co74}, we say that a factor $M$ is \emph{full} if the map $\Ad\colon \mathcal U(M)\to \Aut(M)$ is open on its range. Equivalently, $M$ is full if and only if the quotient topology on $\Inn(M)$ comming from the surjection $\mathcal{U}(M) \rightarrow \Inn(M)$ coincides with the induced topology comming from the inclusion $\Inn(M) \subset \Aut(M)$. In that case $\Inn(M)$ is a complete topological group hence it must be closed in $\Aut(M)$ and the quotient group $\Out(M)$ is also a Hausdorf complete topological group (Polish if $M_*$ is separable).

We also recall \cite{Co74} that a factor $M$ is full if and only if it satisfies the following property: every uniformly bounded net $(x_i)_{i \in I}$ in $\ell^\infty(I,M)$ that is {\em centralizing}, meaning that $\lim_i \|x_i \varphi - \varphi x_i\| = 0$ for all $\varphi \in M_\ast$, must be {\em trivial}, meaning that there exists a bounded net $(\lambda_i)_{i \in I}$ in $\C$ such that $x_i - \lambda_i 1 \to 0$ strongly as $i \to \infty$. See also \cite{Ma18} for another characterization of fullness.

\subsection*{Bimodules and Popa's intertwining theory} Let $M$ and $N$ be two von Neumann algebras. An \emph{$M$-$N$-bimodule} is a $*$-representation $\pi_{\cH} \colon M\odot N\op\rightarrow \mathbf{B}(\cH)$ that is normal on each tensor component, where $\odot$ is the algebraic tensor product and $N\op=\{n\op:n\in N\}$ is the opposite von Neumann algebra of $N$. When the underlying representation $\pi_{\cH}$ is obvious, we will often use the notation ${_M}\cH_N$ to specify the $M$-$N$-bimodule  structure of $\cH$. We refer the reader to the preliminary section of \cite{AD95} for the general theory of bimodules and for the definition of the Connes' fusion tensor product. We will simply fix some notations and recall the needed facts. 

We denote by $\mathcal{L}_{N^{\op}}(\cH)$ the commutant of the right $N$-action on $\cH$. Then $\rL^2(\mathcal{L}_{N^{\op}}(\cH))$ identifies canonically with $\cH \otimes_B \overline{\cH}$ where $\overline{\cH}$ is the opposite $B$-$A$-bimodule of $\cH$. Suppose that $N \subset M$ is a subalgebra of $M$. We denote by $\langle M,N \rangle$ the commutant of the right $N$-action on $\rL^2(M)$ (namely, the restriction of the canonical right $M$-action). Then we have $\rL^2(M) \otimes_N \rL^2(M)=\rL^2 \langle M, N \rangle$ as $M$-$M$-bimodules. We view $M$ as a subalgebra of $\langle M, N \rangle$.

We will say that an $M$-$N$-bimodule $\cH$ is \emph{contained} in another $M$-$N$-bimodule $\cK$, written abusively as $\cH \subset \cK$, if there exists an $M$-$N$-bimodular isometry $V\colon \cH \rightarrow \cK$. We will say that $\cH$ is \emph{weakly contained} in $\cK$, written as $\cH\prec\cK$, if we have $\| \pi_{\cH}(T) \| \leq \| \pi_{\cK}(T) \|$ for all $T \in M \odot N\op$.

We have the following very important characterizations (see \cite[Appendix]{BMO19} for item (ii)):
\begin{thm}\label{bimodule theorem}
Let $M \subset N$ be an inclusion of von Neumann algebras. Then the following holds:
\begin{itemize}
\item ${_M}\rL^2(M)_M \subset  {_M}\rL^2(N)_M$ if and only if there exists a normal conditional expectation from $N$ to $M$.
\item ${_M}\rL^2(M)_M \prec  {_M}\rL^2(N)_M$ if and only if there exists a conditional expectation from $N$ onto $M$.
\end{itemize}
\end{thm}

We now introduce the notion of \emph{left weakly mixing} bimodules and \emph{left amenable} bimodules via the following propositions which are consequences of Theorem \ref{bimodule theorem}.

\begin{prop}[Left weakly mixing bimodules] Let $A$ and $B$ be two von Neumann algebras and let $\cH$ be an $A$-$B$-bimodule. The following properties are equivalent:
\begin{itemize}
\item The $A$-$A$-bimodule $\cH \otimes_B \overline{\cH}$ is disjoint from $\rL^2(A)$, i.e.\ does not contain $z\rL^2(A)$ for any non-zero projection $z \in \mathcal{Z}(A)$. 
\item The $A$-$A$-bimodule $\cH \otimes_B \cK$ is disjoint from $\rL^2(A)$ for every $B$-$A$-bimodule $\cK$.
\item There is no normal conditional expectation $E : z\mathcal{L}_{B^{\op}}(\cH)z \rightarrow zA$ for any non-zero projection $z \in \mathcal{Z}(A)$. 
\end{itemize}
When these properties are satisfied, we say that $\cH$ is left weakly mixing.
\end{prop}

\begin{prop}[Left amenable bimodules] Let $A$ and $B$ be two von Neumann algebras and let $\cH$ be an $A$-$B$-bimodule. The following properties are equivalent:
\begin{itemize}
\item The $A$-$A$-bimodule $\cH \otimes_B \overline{\cH}$ weakly contains $\rL^2(A)$.
\item The $A$-$A$-bimodule $\cH \otimes_B \cK$ weakly contains $\rL^2(A)$ for some $B$-$A$-bimodule $\cK$.
\item There exists a conditional expectation $E : \mathcal{L}_{B^{\op}}(\cH) \rightarrow A$.
\end{itemize}
When these properties are satisfied, we say that $\cH$ is left amenable.
\end{prop}

We recall the following Popa's \textit{intertwining-by-bimodule} technique \cite{Po01,Po03}. For the proof, we refer the reader to \cite[Theorem 4.3]{HI15} and \cite[Theorem 2]{BH16}. Recall that a von Neumann subalgebra $P\subset M$ is \textit{with expectation} if there is a faithful normal conditional expectation from $1_PM1_P $ onto $P$.

\begin{thm}[{\cite{Po01,Po03}}]\label{intertwining theorem}
Let $M$ be any $\sigma$-finite von Neumann algebra and $A \subset 1_AM1_A$ and $B \subset 1_BM1_B$ two von Neumann subalgebras with expectations. Then the following are equivalent:
\begin{itemize}
\item The $A$-$B$-bimodule $1_A\rL^2(M)1_B$ is not left weakly mixing.
\item There exists projections $e\in A$, $f\in B$, a nonzero partial isometry $v\in eMf$ and a unital normal $\ast$-homomorphism $\theta\colon eAe \to fBf$ such that $v\theta(a)= av$ for all $a\in eAe$.
\end{itemize}
When these properties hold, we write $A \prec_M B$.
\end{thm}

We will also write $A \lessdot_M B$ when the $A$-$B$-bimodule $1_A \rL^2(M) 1_B$ is left amenable. When $1_B=1_M$, this means there is conditional expectation from $1_A\langle M,B \rangle 1_A$ onto $A$. If it is normal on $1_AM1_A$, this is equivalent to relative amenability. We will not use this normality in this paper and our notion of $A \lessdot_M B$ is more appropriate for our study.

\begin{prop} \label{convergence weak containment}
Let $P\subset M$ be a von Neumann algebra and $E\colon M\to P$ a normal conditional expectation. 
Suppose that $E_i : M \rightarrow P_i$ is a net of normal conditional expectations on von Neumann subalgebras $P_i \subset M$ which converges to $E$ in the sense that $\|\phi\circ E_i - \phi\circ E\| \to0$ for all $\phi\in M_*$. 
Then $\rL^2(P) \prec \bigoplus_{i \in I} \rL^2 \langle M, P_i \rangle$ as $P$-$P$-bimodules. If moreover $E$ is faithful then we have $\rL^2\langle M, P \rangle \prec \bigoplus_{i \in I} \rL^2 \langle M, P_i \rangle$ as $M$-$M$-bimodules.
\end{prop}
\begin{proof}
Let $q\in P' \cap M$ be the support projection of $E$ and we denote by $r$ the right action of $q$ on $\rL^2(M)$, which is contained in $\langle M,P\rangle$. 
Let $p\in P$ be any $\sigma$-finite projection and $\varphi\in (pPp)_*^+$ a faithful state. Let $\xi \in pr\rL^2 \langle M,P \rangle pr$ be the canonical vector such that 
	$$ \langle x \xi y, \xi \rangle = \langle E(x) \varphi^{\frac{1}{2}}y ,\varphi^{\frac{1}{2}}\rangle, \quad \text{for all }x,y\in pqMpq.$$
Put $\varphi_i= \varphi\circ E_i$ and observe that $\varphi_i \to \varphi$ by assumption. We have
	$$ \langle E(x) \varphi^{\frac{1}{2}}y ,\varphi^{\frac{1}{2}}\rangle = \lim_i\langle E_i(x) \varphi_i^{\frac{1}{2}}y ,\varphi_i^{\frac{1}{2}}\rangle \quad \text{for all }x,y\in pqMpq.$$
This shows that $$pqr\rL^2 \langle M, P \rangle pqr \prec \bigoplus_{i \in I} pq\rL^2\langle M, P_i \rangle pq$$ as $pqMpq$-bimodules. Since the $\sigma$-finite projection $p$ is arbitrary, the case $p=1$ also holds. 
If $q=r=1$, then we are done. For general $q$, since $\rL^2(P)$ is contained in $qr\rL^2 \langle M, P \rangle qr$ as $P$-$P$-bimodules, we are also done.
\end{proof}

\section{Tensor factors and flip maps}

Let $M$ be a factor. A \emph{tensor factor} of $M$ is a subfactor $P \subset M$ such that $M=P \ovt (P' \cap M)$. We denote by $\mathrm{TF}(M)$ the set of all tensor factors of $M$. For $P \in \mathrm{TF}(M)$, we will often denote its commutant by $P^c=P' \cap M \in \mathrm{TF}(M)$ when no confusion is possible. We equip $\mathrm{TF}(M)$ with the weakest topology which makes the maps 
$$ \mathrm{TF}(M) \ni P \mapsto \varphi |_P \otimes \psi |_{P^c} \in M_*$$ continuous for every $\varphi, \psi \in M_*$.

Let $\widehat{M}=M \ovt M$ be the tensor double of $M$. We will often distinguish the two copies of $M$ (and its subalgebras) by using the notation $M_1=M\otimes \C$ and $M_2=\C \otimes M$. We denote by $\sigma_M$ the \emph{flip automorphism} of $\widehat{M}$ given by $\sigma_M(x \otimes y)=y \otimes x$ for every $x,y \in M$. For every $P \in \mathrm{TF}(M)$, we obtain naturally a tensor product decomposition $\widehat{M}=\widehat{P} \ovt \widehat{P^c}$. Therefore, we can view $\sigma_P$ as an automorphism of $\widehat{M}$ by identifying abusively $\sigma_P$ with $\sigma_P \otimes \id_{\widehat{P^c}}$. The map $P \mapsto \sigma_P$ is clearly injective since $P=\{ x \in M \mid \sigma_P(x \otimes 1)=1 \otimes x \}$ for every $P \in \mathrm{TF}(M)$.

\begin{thm}\label{thm for TF(M)}
Let $M$ be a factor. The map $\mathrm{TF}(M) \ni P \mapsto \sigma_P \in \Aut(\widehat{M})$ is a homeomorphism on its range, and its range is closed.
\end{thm}
\begin{proof} 
Let $\iota : (M \ovt M) _* \rightarrow M_*$ be the continuous map given by $\iota(\phi)(x)=\phi(x \otimes 1)$ for all $x \in M$. Let $\varphi \in M_*$. Then, we have $\varphi |_P \otimes \psi |_{P^c}=\iota(\sigma_P( \varphi \otimes \psi))$ for all $\varphi, \psi \in M_*$. This shows that if $\sigma_{P_i} \to \sigma_P$ then $P_i \to P$.

Conversely, suppose that $P_i \to P$. Let $K=\{ \varphi \in M_* \mid \varphi=\varphi|_P \otimes \varphi|_{P^c} \}$. Take $\varphi,\psi \in K$. Then we have $\| \varphi|_{P_i} \otimes \varphi |_{P_i^c} - \varphi \| \to 0$ and similarly for $\psi$. Thus we get
$$ \lim_i \| \sigma_{P_i}(\varphi \otimes \psi) - (\psi|_{P_i} \otimes \varphi |_{P_i^c}) \otimes (\varphi|_{P_i} \otimes \psi |_{P_i^c}) \|=0.$$
But we have
$$ \lim_i  (\psi|_{P_i} \otimes \varphi |_{P_i^c}) \otimes (\varphi|_{P_i} \otimes \psi |_{P_i^c})  =  (\psi|_{P} \otimes \varphi |_{P^c}) \otimes (\varphi|_{P} \otimes \psi |_{P^c}) =\sigma_P(\varphi \otimes \psi).$$
Thus we have the pointwise norm convergence of $\sigma_{P_i}$ to $\sigma_P$ on the set $K \odot K$ which is dense in $(M \ovt M)_*$. We conclude that $\sigma_{P_i}$ converges to $\sigma_P$.

Let us now show that the range is closed. Suppose that a net $(\sigma_{P_i})_{i \in I}$ converges to some $\alpha \in \Aut(\widehat{M})$. Fix $\varphi$ a normal state on $M$. For all $i \in I$, define a normal conditional expectation from $M$ to $P_i$ by $E_{P_i} = \id_{P_i} \otimes (\varphi|_{P_i^c})$. Note that we have $E_{P_i}(x) \otimes 1=(\id \otimes \varphi)(\sigma_{P_i}(1 \otimes x))$ for all $x \in M$. Define a normal completely positive map $E : M \rightarrow M$ by $E(x) \otimes 1=(\id \otimes \varphi)(\alpha(1 \otimes x))$. Observe that $\lim_i E_{P_i}(x)=E(x)$ in the $*$-strong topology for all $x \in M$. Also, observe that for every $\psi \in M_*$, we have $\| \psi \circ E_{P_i} - \psi \circ E \| \leq \| \sigma_{P_i}(\psi \otimes \varphi)-\alpha(\psi \otimes \varphi) \| \to 0$. From this, it is easy to see that in the weak$^*$-topology, we have $E(xE(y))=\lim_i E_{P_i}(xE_{P_i}(y))=\lim_i E_{P_i}(x)E_{P_i}(y)=E(x)E(y)$ for all $x,y \in M$. This shows that $E$ is a normal conditional expectation on a subalgebra $P \subset M$. For every unitary $u \in P$, we have $u \otimes 1 = (\id \otimes \varphi)(\alpha(1 \otimes u))$. This forces $u \otimes 1= \alpha(1 \otimes u)$ because $\id \otimes \varphi$ is a conditional expectation and $u \otimes 1$ and $\alpha(1 \otimes u)$ are unitaries. Thus we have $P=\{ x \in M \mid x \otimes 1=\alpha(1 \otimes x) \}$. We denote $E$ by $E_P$ from now on.

We can do the same with $Q_i=P_i' \cap M$ as $\sigma_{Q_i}=\sigma_M \circ \sigma_{P_i} \to \sigma_M \circ \alpha$. We obtain a subalgebra $Q \subset M$ with a normal conditional expectation $E_Q : M \rightarrow Q$ given by $ E_Q(x) \otimes 1=(\id \otimes \varphi)(\alpha(x \otimes 1))$ for all $x \in M$. We have $Q=\{ x \in M \mid x \otimes 1=\alpha(x \otimes 1) \}$. Thus, we see that $P \otimes \C=\alpha(\C \otimes P)$ and $Q \otimes \C=\alpha(Q \otimes \C)$ are in tensor product position inside $M \otimes \C$. It remains to show that $P$ and $Q$ generate $M$. Let $D \subset M$ be the von Neumann algebra generated by $P$ and $Q$. Observe that $D \otimes \C=\alpha(Q \ovt P)$. Thus $E_{D \otimes \C}=\alpha \circ (E_Q \otimes E_P) \circ \alpha$ defines a normal conditional expectation of $M \ovt M$ onto $D \otimes \C$. Observe that $\psi \circ (E_Q \otimes E_P) = \lim_i \psi \circ (E_{Q_i} \otimes E_{P_i})$ for all $\psi \in M_* \odot M_*$ hence for all $\psi \in (M \ovt M)_*$ by density. Thus we get
$$\psi \circ E_{D \otimes \C}=\lim_i \psi \circ \sigma_{P_i} \circ (E_{Q_i} \otimes E_{P_i}) \circ \sigma_{P_i}$$ for all $\psi \in (M \ovt M)_*$. Since $\sigma_{P_i}(Q_i \otimes P_i)=M \otimes \C$ for all $i$, we have that $\sigma_{P_i} \circ (E_{Q_i} \otimes E_{P_i}) \circ \sigma_{P_i}$ is a conditional expectation onto $M \otimes \C$ for all $i$. We conclude that $E_{D \otimes \C}(x \otimes 1)=x \otimes 1$ for all $x \in M$, i.e.\ $D=M$ as we wanted.
\end{proof}
\begin{cor}
If $M_*$ is separable, then $\mathrm{TF}(M)$ is a Polish space. 
\end{cor}
\begin{proof}
By the theorem, $\mathrm{TF}(M)$ is homeomorphic to a closed subset of $\Aut(\widehat{M})$.
\end{proof}

The following items will be very useful in the study of $\mathrm{TF}(M)$ for a factor $M$. For this, recall that all $\sigma$-finite infinite projections in $M$ are equivalent, so that we can often reduce problems to the $\sigma$-finite case.

\begin{lem}\label{lemma for intertwining and flip}
Let $P,Q \in \mathrm{TF}(M)$. The following conditions are equivalent:
\begin{itemize}
\item  [$(\rm i)$] $P \prec_M Q$;
\item [$(\rm ii)$] $\sigma_{Q}(P_1) \prec_{\widehat{M}} M_2$;
\item [$(\rm iii)$] $Q^c \prec_M P^c$.
\end{itemize}
\end{lem}
\begin{proof}
	By the bimodule definition of the relation $\prec_M$, it is easy to see that $P\prec_MQ$ if and only if $P_1\prec_{\widehat{M}}Q_1\ovt Q_2^c$. Then applying $\sigma_Q$, we get (i) $\Leftrightarrow$ (ii). For item (iii), if $M$ is $\sigma$-finite, the proof is given in \cite[Lemma 4.9]{HI15}. The general case can be reduced to the $\sigma$-finite case.
\end{proof}

\begin{prop}\label{tensor lemma2}
Let $P,Q \in \mathrm{TF}(M)$. The following conditions are equivalent.
\begin{itemize}
\item  [$(\rm i)$] We have $P \prec_M Q$.
\item [$(\rm ii)$] There exists $D \in \mathrm{TF}(Q) \subset \mathrm{TF}(M)$ such that $P \sim D$.
\end{itemize}
\end{prop}
\begin{proof}
If $M$ is $\sigma$-finite, then the proof is given in \cite{OP03}, \cite[Lemma 4.13]{HI15} and \cite[Proposition 7.3]{HMV16}.  The general case can be reduced to the $\sigma$-finite case.
\end{proof}

\begin{prop}\label{tensor lemma for stably unitary conjugacy}
Let $P,Q \in \mathrm{TF}(M)$. The following conditions are equivalent.
\begin{itemize}
\item  [$(\rm i)$] We have $P \prec_M Q$ and $Q \prec_M P$.
\item  [$(\rm ii)$] $P \sim Q$.
\item [$(\rm iii)$] $\sigma_P \circ \sigma_Q \in \Inn(\widehat{M})$.
\end{itemize}
\end{prop}
\begin{proof}
The equivalence of (i) and (ii) is an easy consequence of Proposition \ref{tensor lemma2}. Item (ii) trivially implies item (iii). Conversely, if item (iii) holds, then $P_1 \prec_{\widehat{M}} \sigma_Q\circ \sigma_P(P_1)$, hence $\sigma_Q(P_1) \prec_{\widehat{M}} M_2$. By Lemma \ref{tensor lemma2}, we get $P\prec_MQ$. Similarly we get $Q\prec_MP$ and item (i) holds.
\end{proof}

\section{Weakly bicentralized subalgebras}

In this section, we investigate the following property which plays a key role in our deformation/rigidity arguments. It was already used in \cite{BMO19}.

\begin{df}
Let $M$ be a von Neumann algebra. We say that a subalgebra $P \subset M$ is \emph{weakly bicentralized} in $M$ if ${_M}\rL^2 \langle M,P \rangle_M \prec {_M}\rL^2(M)_M$.
\end{df}
The terminology is justified by the following bicentralizer criterion from \cite{BMO19}.
\begin{prop} \label{weakly bicentralized criterion}
Let $M$ be a von Neumann algebra and let $P$ be a subalgebra of $M$. Suppose that there exists a faithful normal state $\varphi$ on $M$ such that $P$ is globally invariant by $\sigma^\varphi$ and $(P' \cap (M^\omega)_{\varphi^\omega})' \cap M=P$ for some cofinal ultrafilter $\omega$. Then $P$ is weakly bicentralized in $M$.
\end{prop}
\begin{proof}
Let $E_P : M \rightarrow P$ be the unique $\varphi$-preserving conditional expectation. Let $\{x_1, \dots, x_n \}$ be a finite subset of $M$. We will use the notations $\underline{M}=M^{\oplus n}$, $\underline{x}=(x_1,\dots,x_n) \in \underline{M}$, $\underline{\varphi}=\varphi^{\oplus n}$, $\underline{M}^\omega=(M^\omega)^{\oplus n}=(M^{\oplus n})^\omega$. For every finite set $F \subset P$ and every $\varepsilon > 0$, we define $$U_{F,\varepsilon}=\{ u \in \mathcal{U}(M) \mid \sum_{a \in F} \|ua-au\|_\varphi < \varepsilon \text{ and } \| u \varphi - \varphi u \| < \varepsilon \}$$
and we let
$$ C_{F,\varepsilon}=\conv \{ u \underline{x} u ^* \mid u \in U_{F,\varepsilon} \}, \quad C= \bigcap_{F,\varepsilon} C_{F,\varepsilon}.$$
Then since $(M^\omega)_{\varphi^\omega}$ is finite, one can follow the proof of \cite[Lemma 3.4]{BMO19} and get that $E_P$ is the pointwise weak$^*$-limit of convex combinations of inner automorphisms of $M$. Note that the condition $\| u \varphi - \varphi u \| < \varepsilon$ above is used to make a unitary element in $M^\omega_{\varphi^\omega}$.


Now, $\rL^2\langle M,P \rangle$ contains a natural $M$-$M$-cyclic vector $\xi$ which satisfies $\langle x \xi y,\xi \rangle=\langle E_P(x)\varphi^{1/2}y, \varphi^{1/2} \rangle$ for all $x,y \in M$. But we have proved that $\langle x \xi y,\xi \rangle=\langle E_P(x)\varphi^{1/2}y, \varphi^{1/2} \rangle$ can be approximated by convex combinations of $\langle x (u\varphi^{1/2})y, (u\varphi^{1/2}) \rangle$ where $u \in \mathcal{U}(M)$. Thus $_M \rL^2\langle M ,P \rangle_M \prec {_M}\rL^2(M)_M$.
\end{proof}

In the following two propositions, we collect basic properties for weakly bicentralized subalgebras.

\begin{prop} \label{proposition weakly bicentralized}
Let $P, M, N$ be three von Neumann algebras.
\begin{itemize}
\item [$(\rm i)$] If $P \subset M \subset N$, $P$ is weakly bicentralized in $M$ and $M$ is weakly bicentralized in $N$, then $P$ is weakly bicentralized in $N$.
\item [$(\rm ii)$] The algebra $P$ is weakly bicentralized in $M$ if and only if $P \ovt N$ is weakly bicentralized in $M \ovt N$.
\item [$(\rm iii)$] If $M_*$ is separable, then $\mathcal{Z}(M)$ is weakly bicentralized in $M$.
\end{itemize}
\end{prop}
\begin{proof}
$(\rm i)$ Since $P$ is weakly bicentralized in $M$, we have $$\rL^2\langle M, P \rangle=\rL^2(M) \otimes_P \rL^2(M)  \prec \rL^2(M)$$ as $M$-$M$-bimodules. By tensoring on the left and on the right with $\rL^2(N) \otimes_M $ and $ \otimes_M \rL^2(N)$ respectively, we get $\rL^2 \langle N, P \rangle \prec \rL^2 \langle N, M \rangle$ as $N$-$N$-bimodules. On the other hand, since $M$ is weakly bicentralized in $N$, we have $\rL^2 \langle N, M \rangle \prec \rL^2(N)$ as $N$-$N$-bimodules. Thus, we conclude that $\rL^2 \langle N,P \rangle \prec \rL^2(N)$ as $N$-$N$-bimodules.

$(\rm ii)$ This follows from the equality of $(M\ovt N)$-bimodules
$$ {_{M \ovt N}}\rL^2\langle M \ovt N, P \ovt N \rangle_{M \ovt N} = {_M}\rL^2 \langle M, P \rangle_{M} \otimes {_N}\rL^2(N)_N.$$

$(\rm iii)$ Let $M=\int_{\oplus} M_x \rd \mu(x)$ be the desintegration of $M$ into factors.  By item $(\rm ii)$, for any factor $N$, we have that $\mathcal{Z}(M)$ is weakly bicentralized in $M$ if and only if $\mathcal{Z}(M)=\mathcal{Z}(M\ovt N)$ is weakly bicentralized in $M \ovt N$. Thus, by taking $N=R_\infty$ and using \cite[Theorem D]{Ma18} if necessary, we may assume that each $M_x$ is a type $\III_1$ factor with trivial bicentralizer. Take $\varphi$ a faithful normal state on $M$. Then we get $(M^\omega_{\varphi^{\omega}})' \cap M=\mathcal{Z}(M)$ and we conclude by Proposition \ref{weakly bicentralized criterion}.
\end{proof}

Recall that an action $\alpha : \Gamma \curvearrowright B$ of a discrete group $\Gamma$ on a von Neumann algebra $B$ is \emph{centrally free} if for every $g \in \Gamma \setminus \{1\}$ and every nonzero $z \in \mathcal{Z}(B)$, there exists a cofinal ultrafilter $\omega$ and some $b \in B_\omega$ such that $\alpha_g(b)z \neq bz$.

\begin{prop}\label{action lemma}
Let $(B,\varphi)$ be a von Neumann algebra with a faithful normal state $\varphi$ and let $\alpha : \Gamma \curvearrowright (B,\varphi)$ be a state preserving action. Assume either that:
\begin{itemize}
	\item [$(\rm i)$] $\Gamma$ is ICC and $\alpha$ is approximately inner; or
	\item [$(\rm ii)$] $\alpha$ is centrally free.
\end{itemize}Then $B$ is weakly bicentralized in $B \rtimes \Gamma$.
\end{prop}
\begin{proof}
 Let $\rE_B : M \rightarrow B$ be the canonical conditional expectation and use it to extend $\varphi$ to $M$. We will use the criterion of Proposition \ref{weakly bicentralized criterion}. We trivially have $B\subset (B' \cap (M^\omega)_{\varphi^\omega})' \cap M$, so we only have to prove the converse. Observe that
	$$B' \cap (B^\omega)_{\varphi^\omega} \subset B' \cap (B^\omega \rtimes \Gamma)_{\varphi^\omega} \subset B' \cap (M^\omega)_{\varphi^\omega}$$
and hence 
	$$[B' \cap (M^\omega)_{\varphi^\omega}]' \cap M \subset [B' \cap (B^\omega \rtimes \Gamma)_{\varphi^\omega}]'\cap M \subset [B' \cap (B^\omega)_{\varphi^\omega}]' \cap M$$
We have only to show that one of these sets is contained in $B$.

	We first assume that $\Gamma$ is ICC and $\alpha$ is approximately inner. Fix $u_g \in \mathcal{U}((B^\omega)_{\varphi^\omega})$ such that $\alpha_g = \Ad(u_g)$ for all $g\in \Gamma$. 
Observe $u_g^* \lambda_g \in B' \cap (B^\omega \rtimes \Gamma)_{\varphi^\omega}$. Fix any $x\in [B' \cap (B^\omega \rtimes \Gamma)_{\varphi^\omega}]'\cap M$ and let $x= \sum_{g\in \Gamma} x_g \lambda_g$ be the Fourier decomposition in $M$. We have
	$$ \sum_{h\in \Gamma} x_h \lambda_h = x = u_g^* \lambda_g x (u_g^* \lambda_g)^* = \sum_{h\in \Gamma} u_g^* \alpha_g(x_h) \alpha_{ghg^{-1}}(u_g) \lambda_{ghg^{-1}} ,$$
so that $x_{ghg^{-1}} = u_g^* \alpha_g(x_h) \alpha_{ghg^{-1}}(u_g)$ for all $g,h\in \Gamma$. This implies $\|x_{ghg^{-1}}\|_2 = \| x_h\|_2$ for all $g,h\in \Gamma$. Since $\Gamma$ is ICC, we conclude $x_g=0$ for all $g\neq e$ and hence $x\in B$.

	Assume next that $\alpha$ is centrally free. Take any $x\in [B' \cap (B^\omega)_{\varphi^\omega}]' \cap M$ and decompose it as $x=\sum_{g\in \Gamma} x_g \lambda_g$. For any $b\in B' \cap (B^\omega)_{\varphi^\omega}$, by comparing coefficients, it holds that
 $bx_g = x_g \alpha_g(b) = \alpha_g(b)x_g$ for all $g\in \Gamma$. This is equivalent to $(b - \alpha_g(b))x_g=0$, so one has $(b - \alpha_g(b))z=0$, where $z\in \mathcal{Z}(B)$ is the central left support projection of $x_g$. Since this holds for all $b\in B' \cap (B^\omega)_{\varphi^\omega}$, by assumption, we conclude that $x_g= 0$ if $g\neq e$, hence $x\in B$.
\end{proof}

\section{Rigidity for full tensor factors}

	In this section, we study the set $\mathrm{TF}(M)$ by assuming $M$ is a full factor. We particularly prove Theorem \ref{main topological}. We start by recalling the following key property.

\begin{prop}[\cite{BMO19}] \label{full bimodule}
Let $M$ be a full factor. Then every $M$-$M$-bimodule that is weakly equivalent to $\rL^2(M)$ must contain $\rL^2(M)$.
\end{prop}

The next lemma applies in particular to $X=\{Q\}$ when $P \lessdot_M Q$.
\begin{lem} \label{key rigidity lemma}
Let $P \in \mathrm{TF}(M)$ and $X \subset \mathrm{TF}(M)$. Suppose that $P$ is full and $$\rL^2(P) \prec \bigoplus_{Q \in X} \rL^2 \langle M,Q \rangle$$ as $P$-$P$-bimodules. Then there exists $Q \in X$ such that $P \prec_M Q$.
\end{lem}
\begin{proof}
Let $\cH= \bigoplus_{Q \in X} \rL^2 \langle M,Q \rangle$. Since each $Q^c=Q' \cap M$ is a factor, $\C$ is weakly bicentralized in $Q^c$. Combined with  Proposition \ref{proposition weakly bicentralized}, each $Q$ is weakly bicentralized in $M$, hence ${_M}\cH_M \prec \rL^2(M)$. Moreover, ${_P}\rL^2(M)_P$ is a multiple of ${_P}\rL^2(P)_P$ because $P \in \mathrm{TF}(M)$. Hence ${_P}\cH_P \prec \rL^2(P)$. Combining this with the assumption, we get that ${_P}\cH_P$ is weakly equivalent to $\rL^2(P)$. Thus $\rL^2(P) \subset {_P}\cH_P$ by Proposition \ref{full bimodule}. We conclude that $\rL^2(P) \subset {_P}\rL^2\langle M,Q\rangle{_P}$ for some $Q \in X$ and this means that $P \prec_M Q$.
\end{proof}

\begin{example} \label{example full rigidity}
Let $M$ be a factor and $P,Q \in \mathrm{TF}_{\mathrm{full}}(M)$ two full tensor factors such that $P^c$ and $Q^c$ are amenable. Since $Q^c$ is amenable, we have $P \lessdot_M Q$, hence $P \prec_M Q$ by Lemma \ref{key rigidity lemma}. Similarly, we have $Q \prec_M P$. We conclude that $P \sim Q$. This provides a short proof of \cite[Theorem 5.1]{Po06a} and \cite[Theorem E]{HMV16}.
\end{example}

\begin{lem} \label{clopen set}
Let $P \in \mathrm{TF}(M)$. Suppose that $P$ is full. Then the set $U=\{ Q \in \mathrm{TF}(M) \mid P \prec_M Q \}$ is both closed and open.
\end{lem}

\begin{proof}
First, we show that $U$ is a neighborhood of $P$. Suppose, by contradiction, that there exists a net $(Q_i)_{i \in I}$ in $\mathrm{TF}(M)$ which converges to $P$ but such that $Q_i \nprec_M P$ for all $i$. Take $\phi$ a normal state on $M$ and define a normal conditional expectation $\rE_i=\id \otimes \phi|_{Q^c} : M \rightarrow Q_i$. Since $(Q_i)_{i \in I}$ converges to $P$, we have that $\rE_i$ converges to the normal conditional expectation $\rE = \id \otimes \phi|_{P^c} : M \rightarrow P$ pointwisely in the norm of $M_*$. Thus $\rL^2(P) \prec \bigoplus_i \rL^2 \langle M, Q_i \rangle$ as $P$-$P$-bimodules by Proposition \ref{convergence weak containment}. By Lemma \ref{key rigidity lemma}, we conclude that $P \prec_M Q_i$ for some $i \in I$ which is a contradiction. Hence $U$ is a neighborhood of $P$.


Now, we show that $U$ is indeed closed and open. Let $(Q_i)_{i \in I}$ be a net in $\mathrm{TF}(M)$ which converges to $Q$. Since $\sigma_Q \circ \sigma_{Q_i} \to \id$, we have that $(\sigma_{Q} \circ \sigma_{Q_i})(P_1)$ converges to $P_1$ in $\mathrm{TF}(\widehat{M})$. Thus, by the first part of the proof, for $i$ large enough we have $P_1 \prec_{\widehat{M}} (\sigma_{Q} \circ \sigma_{Q_i})(P_1)$, hence $\sigma_Q(P_1) \prec_{\widehat{M}} \sigma_{Q_i}(P_1)$. Similarly, since $\sigma_{Q_i} \circ \sigma_{Q} \to \id$, we get $\sigma_{Q_i}(P_1) \prec_{\widehat{M}} \sigma_{Q}(P_1)$ for $i$ large enough. We conclude that for $i$ large enough we have $\sigma_{Q_i}(P_1) \sim_{\widehat{M}} \sigma_Q(P_1)$. 
Combined with Lemma \ref{lemma for intertwining and flip}, 
$Q \in U$ if and only if $Q_i \in U$ for $i$ large enough,. This means that $U$ is closed and open.
\end{proof}

Recall that $\mathrm{TF_{full}}(M) \subset \mathrm{TF}(M)$ is the set of tensor factors which are full and $\mathrm{TF}(M)$ has an equivalence relation given in Proposition \ref{tensor lemma for stably unitary conjugacy}. Now we prove Theorem \ref{main topological}.

\begin{thm}
Let $M$ be any factor. Then the following holds:
\begin{itemize}
\item [$(\rm i)$] The space $\mathrm{TF_{full}}(M)/{\sim}$ is Hausdorff and totally disconnected.
\item [$(\rm ii)$] If $M$ is full, the space $\mathrm{TF}(M)/{\sim}$ is discrete.
\item [$(\rm iii)$] If $M_*$ is separable, then $\mathrm{TF_{full}}(M)/{\sim}$ is countable.
\end{itemize}
\end{thm}
\begin{proof}
$(\rm i)$ Let $P_1,P_2 \in \mathrm{TF_{full}}(M)$ and suppose that $P_1 \nsim_M P_2$. Assume, without loss of generality, that $P_1 \nprec_M P_2$. Then, by Lemma \ref{clopen set}, $\{ Q \in \mathrm{TF}(M) \mid P_1 \prec_M Q \}$ is an open set which contains $P_1$ and its complement is also an open set which contains $P_2$. This shows that the equivalence classes of $P_1$ and $P_2$ in $\mathrm{TF_{full}}(M)/{\sim}$ are separated by two open sets which form a partition of $\mathrm{TF_{full}}(M)/{\sim}$. Therefore, $\mathrm{TF_{full}}(M)/{\sim}$ is Hausdorff and totally disconnected.

$(\rm ii)$. Let $P \in \mathrm{TF}(M)$. Since $M$ is full, we know that $P$ and $P^c$ are also full. Therefore by Lemma \ref{lemma for intertwining and flip},
$$\{ Q \in \mathrm{TF}(M) \mid Q \sim_M P\}=\{ Q \in \mathrm{TF}(M) \mid P \prec_M Q \text{ and }  P^c \prec_M Q^c\}$$
is an intersection of two open sets by Lemma \ref{clopen set}. This shows that the equivalence classes of $\sim$ are open, which means that the quotient is discrete.

$(\rm iii)$. In this case, $\mathrm{TF}(M)$ is Polish by Theorem \ref{thm for TF(M)}. 
When $M$ is full, the conclusion follows from $(\rm ii)$ as $\mathrm{TF}(M)/{\sim}$ is a discrete separable space, hence countable. Now, for general $M$, the space $\mathrm{TF_{full}}(M)$ is separable so we can find a dense countable subset $X \subset \mathrm{TF_{full}}(M)$. For every $P \in X$, let $U_P=\{ Q \in \mathrm{TF_{full}}(M) \mid Q \prec_M P \}$. Observe that $U_P/{\sim}$ is in bijection with $\mathrm{TF}(P)/{\sim}$, thus it is countable because $P$ is full. Moreover, by using Lemma \ref{clopen set} and since $X$ is dense, we know that $\mathrm{TF_{full}}(M)=\bigcup_{P \in X} U_P$. This shows that $\mathrm{TF_{full}}(M)/{\sim}$ is a countable union of countable sets.
\end{proof}

\section{Application to fundamental groups}

In this section, we study tensor factors $M\ovt N$ by assuming $M$ is a full factor. For this, we will use the topological structure of $\mathrm{TF}(M\ovt N)$ discussed in the last section. We particularly prove Theorem \ref{trace scaling}.

\begin{thm} \label{main open map}
Let $M$ and $N$ be two infinite factors. Suppose that $M$ is full. Then the following map is open:
\begin{gather*} \iota : \;   \mathcal{U}(M \ovt N) \times \Aut(M) \times \Aut(N)  \rightarrow \Aut(M \ovt N) \\
 (u,\alpha, \beta) \mapsto \Ad(u) \circ (\alpha \otimes \beta).
\end{gather*}
In particular, $\Out(M) \times \Out(N)$ is an open subgroup of $\Out(M \ovt N)$.
\end{thm}
\begin{proof}
We have to show that the image by $\iota$ of any neighoborhood of the identity $\mathcal{V} \subset \mathcal{U}(M\ovt N) \times \Aut(M) \times \Aut(N)$ is again a neighborhood of the identity in $\Aut(M \ovt N)$. Let $(\theta_i)_{i \in I}$ be a net in $\Aut(M \ovt N)$ which converges to the identity and let us show that $\theta_i \in \iota(\mathcal{V})$ for $i$ large enough. Since $\theta_i(M)$ converges to $M$ in $\mathrm{TF}(M \ovt N)$, we know by Lemma \ref{clopen set} that $M \prec_{M \ovt N} \theta_i(M)$ for $i$ large enough. By the same argument applied to $\theta_i^{-1}$, we also get $\theta_i(M) \prec_{M \ovt N} M$ for $i$ large enough. Therefore $\theta_i(M) \sim M$, hence also $\theta_i(N) \sim N$ for all $i$ large enough. This shows that $\theta_i$ is eventually of the form $\theta_i = \iota(u_i,\alpha_i,\beta_i)$ for some triples $(u_i,\alpha_i,\beta_i)$ in the domain of $\iota$. Next, we have to show that we can actually take $(u_i,\alpha_i,\beta_i) \in \mathcal{V}$. Observe that the outer automorphism class $[\alpha_i \otimes \beta_i]$ converges to the identity in $\Out(M \ovt N)$. By \cite[Theorem A]{HMV16}, this implies that $([\alpha_i],[\beta_i])$ converges to the identity in $\Out(M) \times \Out(N)$. Therefore, up to replacing $\alpha_i$ by $\Ad(v_i) \circ \alpha_i$, $\beta_i$ by $\Ad(w_i) \circ \beta_i$ and $u_i$ by $u_i(v_i^* \otimes w_i^*)$ for some unitaries $v_i \in \mathcal{U}(M)$ and $w_i \in \mathcal{U}(N)$, we may simply assume that both $\alpha_i$ and $\beta_i$ converges to the identity in $\Aut(M)$ and $\Aut(N)$. Then, in that case, since $\theta_i=\iota(u_i,\alpha_i,\beta_i)$ converges to the identity, we must also have that $\Ad(u_i)$ converges to the identity in $\Aut(M \ovt N)$. Since $M$ is full, by \cite[Theorem A]{HMV16}, there exists $v_i \in \mathcal{U}(N)$ such that $u_i(1 \otimes v_i)^*$ converges to $1$ strongly and $\Ad(v_i)$ converges to the identity in $\Aut(N)$. We conclude that $\theta_i=\iota(u_i(1 \otimes v_i)^*, \alpha_i, \Ad(v_i) \circ \beta_i) \in \iota(\mathcal{V})$ for $i$ large enough.
\end{proof}

\begin{thm} \label{cocycle perturbation}
Let $M$ and $N$ be two infinite factors with separable predual. Suppose that $M$ is full. Let $\theta : \R \curvearrowright M \ovt N$ be a one-parameter group. Then there exists two one-parameter groups $\alpha : \R \curvearrowright M$ and $\beta : \R \curvearrowright N$ such that $\alpha \otimes \beta$ is a cocycle perturbation of $\theta$, i.e.\ there exists a continuous map $u : \R \rightarrow \mathcal{U}(M \ovt N)$ such that 
$$ \forall s,t \in \R, \quad u_{s+t}= u_s\theta_s(u_t),$$
$$ \forall t \in \R, \quad \alpha_t \otimes \beta_t=\Ad(u_t) \circ  \theta_t.$$
\end{thm}
\begin{proof} Since $\R$ is connected and the continuous morphism between Polish groups
\begin{align*} \iota : \quad &  \mathcal{U}(M \ovt N) \times \Aut(M) \times \Aut(N)  \rightarrow \Aut(M \ovt N) \\
& (u,\alpha, \beta) \mapsto \Ad(u) \circ (\alpha \otimes \beta) 
\end{align*}
is open, the image of $\theta$ is contained in the image of $\iota$ and there exists a borel lift
\begin{align*} \rho : \quad &  \R \rightarrow \mathcal{U}(M \ovt N) \times \Aut(M) \times \Aut(N) \\
& t \mapsto (u_{-t},\alpha_t,\beta_t)
\end{align*}
such that $\theta = \iota \circ \rho$. Observe that $\R : t \mapsto [\alpha_t] \in \Out(M)$ is a group morphism. By \cite{Su80}, since $M$ is infinite and the cohomology group $H^3(\R,\T)$ is trivial, we can find a borel map $\lambda \mapsto v_t \in \mathcal{U}(M)$ such that $t \mapsto \Ad(v_t) \circ \alpha_t$ is a continuous group morphism. Therefore, up to replacing $u_{-t}$ by $u_{-t}(v_t^* \otimes 1)$ and $\alpha_t$ by $\Ad(v_t) \circ \alpha_t$, we may assume that $\alpha : \R \ni t \mapsto \alpha_t \in \Aut(M)$ is a continuous group morphism. Similarly, we may assume that $\beta : t \mapsto \beta_t$ is a continuous group morphism. Then we have that $t \mapsto \alpha_t \otimes \beta_t = \Ad(u_t) \circ \theta_t$ is also a group morphism. This implies that $\Ad(u_{s+t})=\Ad(u_s)\Ad(\theta_s(u_t))$ for all $s,t \in \R$. Therefore, $u_{s+t}=\chi(s,t) u_s\theta_s(u_t)$ where $\chi : \R \times \R \rightarrow \T$ is a scalar $2$-cocycle. Since $H^2(\R,\T)$ is trivial, the $2$-cocycle $\chi$ is a coboundary. Thus, we may perturb $u_t$ by scalars in $\T$ so that $t \mapsto u_t$ becomes a true $1$-cocycle.
\end{proof}

\begin{proof}[Proof of Theorem \ref{trace scaling}]
Let $M$ and $N$ be two factors of type $\II_\infty$ with separable predual and suppose that one of them is full.

$(\rm i)$ There are two surjective maps 
\begin{align*}
	&\mathrm{Out}(M \ovt N) \to \mathcal{F}(M\ovt N) ;\\
	&\mathrm{Out}(M) \times \mathrm{Out}( N) \to \mathcal{F}(M)\mathcal{F}(N).
\end{align*}
They induce a surjective map from $\mathrm{Out}(M \ovt N )/\mathrm{Out}(M) \times \mathrm{Out}(N)$ onto $\mathcal{F}(M\ovt N) / \mathcal{F}(M)\mathcal{F}(N)$. By Theorem \ref{main open map}, $\mathrm{Out}(M \ovt N )/\mathrm{Out}(M) \times \mathrm{Out}(N)$ is discrete hence countable (because $M_*$ and $N_*$ are separable) and we get the conclusion.

$(\rm ii)$ Let $\theta : \R^*_+ \rightarrow \Aut(M \ovt N)$ be a trace scaling action. Then, by Theorem \ref{cocycle perturbation}, we can find two actions $\alpha : \R^*_+ \rightarrow \Aut(M)$ and $\beta : \R^*_+ \rightarrow \Aut(N)$ such that $(\alpha_\lambda \otimes \beta_\lambda) \circ \theta_{\lambda}^{-1}$ is inner for all $\lambda > 0$. In particular, we have $\mathrm{Mod}(\alpha_\lambda)\mathrm{Mod}(\beta_\lambda)=\mathrm{Mod}(\theta_\lambda)=\lambda$ for all $\lambda > 0$. Since $\lambda \mapsto \mathrm{Mod}(\alpha_\lambda)$ is a group homomorphism, there must exist some $s \in \R$ such that $\mathrm{Mod}(\alpha_\lambda)=\lambda^s$ hence $\mathrm{Mod}(\beta_\lambda)=\lambda^{1-s}$ for all $\lambda > 0$. We conclude that $M$ admits a trace scaling action (if $s \neq 0$) or $N$ admits a trace scaling action (if $s \neq 1$).
\end{proof}

\section{Noncommutative Bernoulli shifts}
	In this section, we investigate the structure of full factors arising from  Bernoulli actions. For this, we first observe that well-known arguments in the deformation/rigidity theory for Bernoulli actions (mostly established in \cite{Po03}) also work in the type III setting. We will then prove Theorem \ref{bernoulli strongly prime} and \ref{bernoulli remembers groups}.

	We first prove the following rigidity results for Bernoulli actions. Recall that our definition of $A\prec_MB$ coincides with the usual one if $M$ is $\sigma$-finite and $A,B\subset M$ are with expectation. 

\begin{thm}\label{rigidity for Bernoulli shift}
Let $\alpha : \Gamma \curvearrowright (B_0,\varphi_0)^{\ovt \Gamma} =:(B,\varphi)$ be a noncommutative Bernoulli shift. Let $N$ be any $\sigma$-finite von Neumann algebra and put $M:=N \ovt (B\rtimes \Gamma)$. Let $p\in M$ be a projection and $P,Q \subset pMp$ von Neumann algebras with expectation such that $P$ and $Q$ are commuting and that $Q$ is finite. Then one of the following conditions holds:
\begin{itemize}
	\item [$(\rm i)$] $Q \prec_M N\ovt L(\Gamma)$; 
	\item [$(\rm ii)$] $Q \prec_M N\ovt B_0^F$ for some finite subset $F\subset \Gamma$; or
	\item [$(\rm iii)$] $P \lessdot_M N\ovt B$.
\end{itemize}
\end{thm}
\begin{proof}
	We only give a sketch of the proof. Following \cite{Io06} (see also \cite[Section 1]{CPS11} and \cite[Section 5]{Ma16}), we define a von Neumann algebra $\widetilde{M}$ and its deformations $(\theta_t,\beta)$. 
We then apply the proof of \cite[Theorem 4.2]{Ma16} to the finite algebra $Q$, and we get either:
\begin{enumerate}
	\item $P' \cap p\widetilde{M}^\omega p\not\subset pM^\omega p$ for some ultrafilter $\omega \in \beta \N$; or
	\item $(\theta_t)_t$ converges uniformly on $(Q)_1$ (in the $\ast$-strong topology) and $Q\prec_{\widetilde{M}} \theta_1(Q)$.
\end{enumerate}
Following the proof of \cite[Lemma 5.1]{Ma16}, the second condition directly implies $(\rm i)$ or $(\rm ii)$.

	We next consider the case that $P' \cap p\widetilde{M}^\omega p\not\subset pM^\omega p$. Then, by the proof of \cite[Lemma 4.1]{Ma16}, we have as $P$-$P$-bimodules
		$$ z\rL^2(P) \prec \rL^2(\widetilde{M}) \ominus \rL^2(M)$$
for some nonzero projection $z \in \mathcal{Z}(P)$. Recall that we have a decomposition as $M$-$M$-bimodules (see \cite[Theorem 5.2]{Ma16})
	$$\rL^2(\widetilde{M})\ominus \rL^2(M) = \bigoplus_i \rL^2\langle M,B_i \rangle,$$
	where each $B_i$ is of the form that $B_i = (N\ovt B_0^{F_i^c})\rtimes \Gamma_i$ for some finite $F_i \subset \Gamma$ and finite group $\Gamma_i\leq \Gamma$. Since $\Gamma_i$ is amenable, it holds as $M$-$M$-bimodules that
	$$\rL^2(\widetilde{M})\ominus \rL^2(M) \prec  \bigoplus_i \rL^2 \langle M, N\ovt B_0^{F_i^c} \rangle.$$
Thus we obtain $z\rL^2(P) \prec  \bigoplus_i \rL^2 \langle M, N\ovt B_0^{F_i^c} \rangle$ as $P$-$P$-bimodules. This means that there exists a conditional expectation from $zLz$ onto $zP$ where $L= \bigoplus_i \langle M, N\ovt B_0^{F_i^c} \rangle$ and $P$ is embedded diagonally in $L$. Since $\langle M , N \ovt B \rangle$ embeds diagonally in $L$, we can restrict it to a conditional expectation from $z \langle M, N \ovt B \rangle z$ on $Pz$. We conclude that $P \lessdot_M N \ovt B$.
\end{proof}

The following two lemma are useful to control normalizers in Bernoulli shift von Neumann algebras.

\begin{lem}\label{mixing lemma for Bernoulli shift1}
	Keep the notation $M=(N\ovt B)\rtimes \Gamma$ as in Theorem \ref{rigidity for Bernoulli shift}. Let $C_0 \subset B_0$ be a von Neumann subalgebra (possibly trivial) with expectation which is globally preserved by $\sigma^{\varphi_0}$ and put $C:=\ovt_{\Gamma}(C_0,\varphi_0)\subset B$.  Let $p\in M$ be a projection and $P \subset pMp$ a von Neumann subalgebra with expectation. The following assertions hold true.
\begin{itemize}
	\item [$(\rm i)$] Let $F\subset \Gamma$ be a finite set and assume that $p\in N\ovt B_0^F$, $P \subset p(N\ovt (C \vee B_0^F))p$, and $P\not\prec_{N\ovt (C \vee B_0^F)}N\ovt (C\vee B_0^E)$ for all genuine subsets $E\subset F$. Then any $x\in pMp$ such that $xa = \beta(a)x$ for all $a\in P$ for some $\beta\in \Aut(P)$, is contained in $(N\ovt B)\rtimes \Gamma_F$, where $\Gamma_F:=\{g\in \Gamma\mid gF=F\}$.

	\item [$(\rm ii)$] Assume that $P \prec_M N\ovt (C\vee B_0^F)$ for a finite set $F\subset \Gamma$ and that  $P\not\prec_M N \ovt C$. Then we have $\mathcal{N}_{qMq}(Pq)'' \prec_M N\ovt B$ for some projection $q\in P' \cap pMp$.

\end{itemize}
\end{lem}
\begin{proof}
For simplicity, we will write $D^F := C\vee B_0^F$ for all $F\subset \Gamma$. 

	$(\rm i)$ Take $x\in pMp$ as in the statement and let $x=\sum_{g\in \Gamma}x_g \lambda_g\in (N\ovt B)\rtimes \Gamma$ be the Fourier decomposition. By comparing coefficients, it holds that 
	$$x_g \alpha_g(a) = \beta(a)x_g, \quad \text{for all } a\in P, \ g\in \Gamma.$$
Fix $g\in \Gamma$ and we prove that if $x_g\neq 0$, then $F=gF$. If $x_g \neq 0$, then one has $P=\beta(P)\prec_{N\ovt B} \alpha_g(P)$. 
Since $P\subset N\ovt D^F$, this implies $P \prec_{N\ovt B} N\ovt D^{gF}$. Thus by our definition of $\prec$, we have
	$$P \prec_{N\ovt D^F} N\ovt D^{F\cap gF}.$$
By the assumption of $P$, this implies $F\cap gF = F$, hence $F=gF$. This finished the proof of item $(\rm i)$.

	$(\rm ii)$ Fix a finite set $F \subset \Gamma$ such that $P \prec_M N\ovt D^F$ and take $(H,f,\pi,w)$ as in \cite[Lemma 2.6]{Is19}. Write $\B=\B(H)$ for simplicity. We may assume the support of $E_{N\ovt D^F \ovt \B}(w^*w)$ is $f$. 
Assume that there is a genuine subset $E\subset F$ such that $\pi(P)\prec_{N\ovt D^F \ovt \B} N\ovt D^E \ovt \B$. Then by the choice of $f$, this implies $P \prec_M N\ovt D^E$. In this case, we can replace $F$ by the smaller set $E$. By continuing this procedure, we can finally find $F$ such that $P \prec_M N\ovt D^F$ with $(H,f,\pi,w)$ such that $\pi(P)\not\prec_{N\ovt D^F \ovt \B} N\ovt D^E \ovt \B$ for all genuine subsets $E \subset F$. 

	In this setting we can apply item $(\rm i)$  to the inclusion $\pi(P)\subset N\ovt D^F \ovt \B$ (by regarding $N\ovt \B$ as $N$). Write $f_0=w^*w \in \pi(P)' \cap f(M\ovt \B)f$ and $e_0 \otimes e_{1,1} =ww^*  \in ({P}' \cap pMp) \otimes \C e_{1,1}$ (where $e_{1,1}\in \B$ is a minimal projection), and observe that $\Ad(w^*) : e_0(M \ovt \B) e_0 \rightarrow f_0(M \ovt \B)f_0$ sends $Pe_0$ onto $\pi(P)f_0$. Therefore, we have
	$$w^*[\mathcal{N}_{e_0{M}e_0}({P}e_0) \otimes \C e_{1,1}] w = \mathcal{N}_{f_0({M}\ovt \B)f_0}(\pi(P)f_0).$$
Using item $(\rm i)$ , it is easy to see that the right hand side of this equation is contained in $(N\ovt B) \rtimes \Gamma_F \ovt \B$. We obtain that 
	$$\mathcal{N}_{e_0{M}e_0}({P}e_0)''\prec_{M} (N\ovt B) \rtimes \Gamma_F.$$
Finally, since $\Gamma_F$ is a finite group by assumption, we conclude that $$\mathcal{N}_{e_0{M}e_0}({P}e_0)''\prec_{M} N\ovt B.$$\end{proof}

\begin{lem}\label{mixing lemma for Bernoulli shift2}
	Keep the notation $M=(N\ovt B)\rtimes \Gamma$ as in Lemma \ref{rigidity for Bernoulli shift}. 
Let $p\in M$ be a projection and $P\subset pMp$ von Neumann subalgebras with expectations. The following assertions hold true.
\begin{itemize}
	\item [$(\rm i)$] Assume that $p\in N \rtimes \Gamma$ and $P \subset p(N\rtimes \Gamma )p$. If $P \nprec_{N\rtimes \Gamma}N$, then one has $\mathcal{N}_{pMp}(P) \subset  N \rtimes \Gamma$.

	\item [$(\rm ii)$] If $P\prec_M N\rtimes \Gamma$ and $P\nprec_{M}N$, then one has $\mathcal{N}_{qMq}(Pq)'' \prec_M N\rtimes \Gamma$ for some projection $q\in P' \cap pMp$.

\end{itemize}
\end{lem}
\begin{proof}
$(\rm i)$ Up to replacing $P$ by $\widetilde{P}=P \oplus p^{\perp}(N \rtimes \Gamma)p^{\perp}$, we may assume that $p=1$. We only have to show that the $P$-$P$-bimodule $\rL^2(M) \ominus \rL^2(N \rtimes \Gamma)$ is weakly mixing. Obseve that the $L(\Gamma)$-bimodule $\rL^2(B \rtimes \Gamma) \ominus \rL^2(L(\Gamma))$ is a multiple of the coarse $L(\Gamma)$-bimodule. Thus the $(N \rtimes \Gamma)$-bimodule $\rL^2(M) \ominus \rL^2(N \rtimes \Gamma)$ is a multiple of the $(N \rtimes \Gamma)$-bimodule $\rL^2 \langle N \rtimes \Gamma, N \rangle$. Since $P \nprec_{N \rtimes \Gamma} N$, the $P$-$P$-bimodule $\rL^2 \langle N \rtimes \Gamma, N \rangle$ is weakly mixing. Thus the $P$-$P$-bimodule $\rL^2(M) \ominus \rL^2(N \rtimes \Gamma)$ is also weakly mixing.

$(\rm ii)$ This follows in a similar way to the proof of Lemma \ref{mixing lemma for Bernoulli shift1}.$(\rm ii)$. 
Take $(H,f,\pi,w)$ as in \cite[Lemma 2.6]{Is19} for $P\preceq_M N \rtimes \Gamma$, and we may assume that 
$\pi(P)\nprec_{N \rtimes \Gamma \ovt \B} N \ovt \B$, where $\B=\B(H)$. 
By item $(\rm i)$, we have
	$$\mathcal{N}_{f(M \ovt \B)f}(\pi( P )) \subset f(N\rtimes \Gamma \ovt \B)f.$$
Since $w^*w \in \pi(P)' \cap f({M}\ovt \B)f\subset  f(N \rtimes \Gamma \ovt \B)f$, we can assume $w^*w = f$. 
Putting $ e_0\otimes e_{1,1} =ww^* \in (P' \cap pMp) \otimes \C e_{1,1}$, one has 
	$$ w^* [ \mathcal{N}_{e_0Me_0}(Pe_0)  \otimes \C e_{1,1}] w  = \mathcal{N}_{f(M \ovt \B)f}(\pi(P))  \subset f((N \rtimes \Gamma) \ovt \B)f. $$
This implies the conclusion.
\end{proof}

	We next show that the sufficient condition in Proposition \ref{action lemma} is easily verified for Bernoulli actions.

\begin{prop} \label{bernoulli weakly bicentralized}
Let $\alpha : \Gamma \curvearrowright (B_0,\varphi_0)^{\ovt \Gamma} =:(B,\varphi)$ be a noncommutative Bernoulli shift where $\Gamma$ is infinite and $B_0$ is nontrivial. Then $\alpha$ is centrally free. In particular, for any subset $F \subset \Gamma$, the subalgebra $\mathcal{Z}(B) \vee B_0^{F}$ is weakly bicentralized in $B \rtimes \Gamma$.
\end{prop}
\begin{proof}
The central freeness of $\alpha$ is obvious if $(B_0)_{\varphi_0} \neq \C$. Suppose now that $(B_0)_{\varphi_0}=\C$ (this forces $B_0$ to be a type $\III_1$ factor). Take $g \in \Gamma \setminus \{ 1 \}$. Let $(h_n)_{n \in \N}$ a sequence in $\Gamma$ which goes to infinity, then we can find a sequence of unitaries $u_n \in \alpha_{h_n}(B_0)$ such that $\varphi(u_n)=0$ and $\|u_n \varphi-\varphi u_n\| \leq \frac{1}{n}$. Then $u=(u_n)^\omega \in B_\omega$ for $\omega \in \beta\N \setminus \N$ and we have $\varphi(\alpha_g(u)u^*)=0$. Thus $\alpha$ is centrally free. By Proposition \ref{action lemma}, we then have that $B$ is weakly bicentralized in $M$ and thanks to Proposition \ref{proposition weakly bicentralized}, we conclude that $\mathcal{Z}(B) \vee B_0^F$ is weakly bicentralized in $M$ for every subset $F \subset \Gamma$.
\end{proof}

\begin{lem} \label{trivial bicentralizer}
Let $M$ and $N$ be factors with separable predual. Assume that $M$ is a type $\III_1$ factor with trivial bicentralizer. Then $M \ovt N$ is a type $\III_1$ factor with trivial bicentralizer.
\end{lem}
\begin{proof}
By \cite[Proposition 7.1]{AHHM18}, the bicentralizer flow of $M \ovt N$ is trivial. By \cite[Theorem D]{Ma18}, we conclude that $M \ovt N$ has trivial bicentralizer.
\end{proof}

Now we can prove Theorem \ref{bernoulli strongly prime}.

\begin{proof}[Proof of Theorem \ref{bernoulli strongly prime}]
Let $M=(N \ovt B) \rtimes \Gamma$  be as in the last statement in Theorem \ref{bernoulli strongly prime}. Suppose that $M=P \ovt Q$ and observe that $P$ and $Q$ are full. We have to show that $B \rtimes \Gamma \prec_M P$ or $B \rtimes \Gamma \prec_M Q$. Let $K$ be a full type $\III_1$ factor with trivial bicentralizer (e.g.\ a free Araki-Woods factor). By Lemma \ref{trivial bicentralizer}, up to replacing $P,Q,N$ by $P \ovt K,Q \ovt K, N\ovt K \ovt K$ in $M \ovt K \ovt K$ respectively, we may assume that $P$ and $Q$ have trivial bicentralizers. Then we can find irreducible finite subfactors with expectation $P_0 \subset P$ and $Q_0 \subset Q$. By Theorem \ref{rigidity for Bernoulli shift}, we know that one of the following conditions holds:
\begin{itemize}
	\item [$(\rm i)$] $Q_0 \prec_M N\ovt L(\Gamma)$; 
	\item [$(\rm ii)$] $Q_0 \prec_M N\ovt B_0^F$ for some finite subset $F\subset \Gamma$; or
	\item [$(\rm iii)$] $P \lessdot_M N\ovt B$.
\end{itemize}
If $P_0 \prec_M N$ or $Q_0 \prec_M N$ then, by taking the commutants, we get $B \rtimes \Gamma \prec_M P$ or $B \rtimes \Gamma \prec_M Q$ and we are done. So we may assume, for the sake of a contradiction, that $P_0 \nprec_M N$ and $Q_0 \nprec_M N$. Then, if one of the conditions $(\rm i)$ or $(\rm ii)$ is satisfied, we can apply Lemma \ref{mixing lemma for Bernoulli shift1} or Lemma \ref{mixing lemma for Bernoulli shift2} to get $P \prec_M N \ovt L(\Gamma)$ or $P \prec_M N \ovt B$. If $P \prec_M N \ovt L(\Gamma)$, then by Lemma \ref{mixing lemma for Bernoulli shift2}, we get $M \prec_M N \ovt L(\Gamma)$ which is not possible because $B_0$ is non-trivial. If $P \prec_M N \ovt B$, then a fortiori, we have $P \lessdot_M N \ovt B$. Thus it only remains to deal with the case where condition $(\rm iii)$ holds. 

Now we assume condition (iii) holds. Since $B$ is the increasing union of $\mathcal{Z}(B) \vee B_0^F (=:D^F)$ over finite subsets $F \subset \Gamma$, Proposition \ref{convergence weak containment} shows, as $M$-$M$-bimodules $$\rL^2 \langle M, N \ovt B \rangle  \prec \bigoplus_{F} \rL^2\langle M, N \ovt D^F \rangle.$$ Since $P \prec_M^w N \ovt B$, we get as $P$-$P$-bimodules $$\rL^2(P) \prec \bigoplus_{F} \rL^2 \langle M, N \ovt D^F \rangle.$$
But thanks to Proposition \ref{bernoulli weakly bicentralized}, we also have as $M$-$M$-bimodules
$$ \rL^2 \langle M, N \ovt D^F \rangle \prec \rL^2(M).$$
Since $P$ is full, we conclude that $P \prec_M N \ovt D^F$ for some finite subset $F \subset \Gamma$ (e.g.\ the proof of Lemma \ref{key rigidity lemma}). 
Therefore, by Lemma \ref{mixing lemma for Bernoulli shift1}, we must have $P \prec_M N \ovt \mathcal Z(B)$. Since $\mathcal Z(B)$ is amenable, we get $P \lessdot_M N$. Finally, since $P$ and $N$ are tensor factors, Lemma \ref{key rigidity lemma} is applied and we conclude $P \prec_M N$.
\end{proof}

We next prove Theorem \ref{bernoulli remembers groups}. We prepare a lemma.

\begin{lem} \label{conjugacy bernoulli}
Let $A$ and $B$ be two $\sigma$-finite factors and let $G \curvearrowright A$ and $H \curvearrowright B$ be two outer actions of discrete groups $G$ and $H$. Suppose that $M=A \rtimes G= B \rtimes H$ with $ A \prec_M B$. Then there exists a unitary $u \in \mathcal{U}(M)$, a normal subgroup $G_0 \lhd G$ and a finite normal subgroup $H_0 \lhd H$  such that $u(A \rtimes G_0)u^*=B \rtimes H_0$. If we also have $B \prec_M A$, then $G_0$ is also finite.
\end{lem}
\begin{proof}
By \cite[Proposition 4.4]{Is19}, we can find a unitary $u \in \mathcal{U}(M)$ such that $uAu^* \subset B \rtimes H_0$ for some finite normal subgroup $H_0 \lhd H$. Since $A \subset A \rtimes G$ has the intermediate subfactor property, we know that $u^*(B \rtimes H_0)u=A \rtimes G_0$ for some subgroup $G_0 < G$. Since $B \rtimes H_0$ is regular in $M$, the subgroup $G_0$ must be normal in $G$. If we also assume that $B \prec_M A$, then $B \rtimes H_0 \prec_M A$ because $H_0$ is finite. Thus we get $A \rtimes G_0 \prec_M A$ and this forces $G_0$ to be finite.
\end{proof}

\begin{proof}[Proof of Theorem \ref{bernoulli remembers groups}]
Let $M=A \rtimes G = B \rtimes H$. Let $Q$ be a diffuse finite subalgebra with expectation in $A_0' \cap A$. By, Theorem \ref{rigidity for Bernoulli shift}, we have either
\begin{itemize}
	\item [$(\rm i)$]  $Q \prec_M L(H)$; 
	\item [$(\rm ii)$] $Q \prec_M B_0^F$ for some finite subset $F\subset H$; or
	\item [$(\rm iii)$] $A_0 \lessdot_M B$.
\end{itemize}
In the case $(\rm i)$, by applying Lemma \ref{mixing lemma for Bernoulli shift2} three times we get $A_0 \prec_M L(H)$, then $A \prec_M L(H)$ and finally $M \prec_M L(H)$ which is not possible. In the case $(\rm ii)$, by applying Lemma \ref{mixing lemma for Bernoulli shift1}, we get $A_0 \prec_M B$ which implies condition $(\rm iii)$. Finally, assume $(\rm iii)$ holds. Then, since $B$ is the increasing union of $B_0^F$ over finite subsets $F \subset H$, Proposition \ref{convergence weak containment} implies that $\rL^2(A_0) \prec  \mathcal{H}=\bigoplus_{F} \rL^2 \langle M, B_0^F \rangle$ as $A_0$-$A_0$-bimodules.

On the other hand, by Proposition \ref{bernoulli weakly bicentralized}, we have $ \rL^2 \langle M, B_0^F \rangle \prec \rL^2(M)$ as $M$-$M$-bimodules. Moreover, $_{A_0}\rL^2(A)_{A_0}$ is a multiple of $\rL^2(A_0)$ while $_{A_0}(\rL^2(M) \ominus \rL^2(A))_{A_0}$ is a multiple of the coarse bimodule $\rL^2(A_0) \otimes \rL^2(A_0)$. This shows that $_{A_0}\rL^2(M)_{A_0} \prec \rL^2(A_0)$. Thus we have $ \rL^2 \langle M, B_0^F \rangle \prec \rL^2(A_0)$ as $A_0$-$A_0$-bimodules for every finite subset $F \subset H$.

Therefore we have showed that $\bigoplus_{F} \rL^2 \langle M, B_0^F \rangle$ is weakly equivalent to $\rL^2(A_0)$ as an $A_0$-$A_0$-bimodule. Since $A_0$ is a full factor, Proposition \ref{full bimodule} implies that $A_0 \prec_M B_0^F$ for some finite subset $F \subset H$. By Lemma \ref{mixing lemma for Bernoulli shift1}, we conclude that $A \prec_M B$. Similarly, we have $B \prec_M A$ and we can therefore apply Lemma \ref{conjugacy bernoulli}.
\end{proof}

\section{Compact actions of higher rank lattices}

In this section, we prove Theorem \ref{irreducible lattice thm} and \ref{compact minimal rigidity}. 
We first translate the unique prime factorization property, using the flip map $\sigma_P$ on the double $\widehat{M}$.

\begin{prop}\label{UPF prop}
Let $\mathcal{C}$ be a set of factors. Then the following are equivalent:
\begin{itemize}
\item [$(\rm i)$] Every $P \in \mathcal{C}$ is prime and for every finite family $P_1, \dots, P_n \in \mathcal{C}$, the factor $M=P_1 \ovt \cdots \ovt P_n$ has the Unique Prime factorization property.
\item [$(\rm ii)$] For every finite family $P_1, \dots, P_n \in \mathcal{C}$, and every automorphism $\alpha$ of the factor $M=P_1 \ovt \cdots \ovt P_n$, there exists a permutation $\sigma$ of $\{1, \dots, n\}$ such that $\alpha(P_i) \sim_M P_{\sigma(i)}$ for all $i \in \{1, \dots, n \}$.
\end{itemize}
\end{prop}
\begin{proof}
$(\rm i) \Rightarrow (\rm ii)$ is obvious. Let us prove the other direction. Assume that $(\rm ii)$ holds. Let $P \in \mathcal{C}$ and take $Q \in \mathrm{TF}(P)$. Then by $(\rm ii)$, the automorphism $\sigma_Q$ of $P \ovt P$ satisfies $\sigma_Q(P \otimes 1) \sim_{\widehat{P}} 1 \otimes P$ or $\sigma_Q(P \otimes 1) \sim_{\widehat{P}} P \otimes 1$. Applying Lemma \ref{lemma for intertwining and flip}, in the first case, we get $Q^c \prec_P \C$ and in the second case we get $Q \prec_P \C$. Thus, $Q$ or $Q^c$ is of type $\mathrm{I}$. This shows that $P$ is prime. Now, consider $P_1, \dots, P_n \in \mathcal{C}$ and $M=P_1 \ovt \cdots \ovt P_n$ and take $Q \in \mathrm{TF}(M)$. By $(\rm ii)$, for every $i$, we must have $\sigma_Q(P_i) \sim_{M\ovt M} P_j \otimes 1$ or $\sigma_Q(P_i) \sim_{M\ovt M} 1 \otimes P_j$ for some $j \in \{1, \dots, n \}$. In the first case, we get $P_i \prec_M Q^c$ and in the second case we get $P_i \prec_M Q$. We conclude that $M$ has the UPF property.
\end{proof}

In what follows, by \emph{higher rank} irreducible  lattice we mean an irreducible lattice $\Gamma < G$ where $G$ is a connected semisimple Lie group with finite center such that every simple quotient of $G$ has real rank $\geq 2$. It is known that such a lattice $\Gamma$ has property (T) and satisfies the conclusion of Margulis' normal subgroup theorem, i.e.\ any normal subgroup $N < \Gamma$ is either finite (and contained in the center) or has finite index in $\Gamma$. 

We will need the following elementary lemma. Recall that two subgroups $H_1,H_2$ of a same group $H$ are \emph{commensurable} if $H_1 \cap H_2$ has finite index in both $H_1$ and $H_2$.

\begin{lem} \label{permutation lattice}
Let $L_1,\dots,L_n, R_1,\dots,R_n$ be irreducible higher rank lattices. Let $H < L_1 \times \cdots \times L_n$ and $K < R_1 \times \dots \times R_n$ be two finite index subgroups and $\phi : H \rightarrow K$ an isomorphism. Then there exists a permutation $\sigma$ of $\{1, \dots, n \}$ such that $\phi(L_i \cap H)$ and $R_{\sigma(i)}$ are commensurable for all $i \in \{1, \dots, n\}$.
\end{lem}
\begin{proof}
We proceed by induction. The result is obvious for $n=1$. Let $n\geq 2$ and suppose that we have proved the result for $n-1$. For each $i$, let $\pi_i$ be the projection on $R_i$. Observe that $L_n \cap H$ has finite index in $L_n$. In particular, $L_n \cap H$ is infinite. Thus there exists $i$ such that $\pi_i(\phi(L_n \cap H))$ is infinite. Assume, without loss of generality, that $i=n$. Note that $\phi(L_n \cap H)$ is a normal subgroup of $K$. Thus $\pi_n(\phi(L_n \cap H))$ is a normal subgroup of $\pi_n(K) \subset R_n$. But $\pi_n(K)$ is an irreducible lattice because it has finite index in $R_n$. Therefore, since $\pi_n(\phi(L_n \cap H))$ is infinite, it must actually have finite index in $\pi_n(K)$, hence also in $R_n$. But, if we let $H'=(L_1 \times \dots \times L_{n-1}) \cap H$, then $\pi_n(\phi(H'))$ is also a normal subgroup of $\pi_n(K)$ which commutes with $\pi_n(\phi(L_n \cap H))$. Thus $\pi_n(\phi(H'))$ is finite. Let $H'' \subset H'$ be the kernel of $\pi_n \circ \phi |_{H'}$. We have that $H''$ is a finite index subgroup of $L_1 \times \dots \times L_{n-1}$ and $\phi(H'') \subset R_1 \times \dots \times R_{n-1}$. Therefore, we can apply the induction hypothesis and wet get that $\phi(H'' \cap L_i)$ and $R_{\sigma(i)}$ are commensurable for some permutation $\sigma$ of $\{1, \dots, n-1\}$. Since $H''$ has finite index in $H$, we actually have that $\phi(H \cap L_i)$ and $ R_{\sigma(i)}$ are commensurable. It only remains to show that $\phi(H \cap L_n)$ and $R_n$ are commensurable.

We know that $\pi_i(\phi(H \cap L_n))$ is finite for all $i \leq n-1$ because it is a normal subgroup of $\pi_i(K)$ and it commutes with $\phi(H \cap L_i) \cap R_i$ which has finite index in $R_i$. Thus the kernel of $\pi_i|_{\phi(H \cap L_n)}$ has finite index in $\phi(H \cap L_n)$ for all $i \leq n-1$. We deduce that the intersection of all this kernels, which is precisely $\phi(H \cap L_n) \cap R_n$, has finite index in $\phi(H \cap L_n)$. It also has finite index in $R_n$ because it is an infinite normal subgroup of $K \cap R_n$. We conclude that $\phi(L_n \cap H)$ and $R_n$ are commensurable.
\end{proof}
\begin{proof}[Proof of Theorem \ref{irreducible lattice thm}]
Let $M_i=A_i \rtimes \Gamma_i$, $A=A_1 \ovt \cdots \ovt A_n$, $\Gamma=\Gamma_1 \times \cdots \times \Gamma_n$ and $M=M_1 \ovt \cdots \ovt M_n=A \rtimes \Gamma$. Let $\alpha$ be an automorphism of $M$. 
By Proposition \ref{UPF prop}, we have only to show that there exists a permutation $\sigma$ of $\{1, \dots, n\}$ such that $\alpha(M_i) \sim_M M_{\sigma(i)}$ for all $i \in \{1, \dots, n \}$. By \cite[Theorem 1.4]{BIP18}, up to composing $\alpha$ by an inner automorphism, we may assume that $\alpha(A)=A$. By \cite[Theorem A]{Io08} (see also \cite[Theorem 5.21]{Fu09}), up to composing $\alpha$ by $\Ad(u)$ for some $u \in \mathcal{N}_M(A)$, we may further assume that $\alpha$ induces a virtual self-conjugacy of the action $\Gamma \curvearrowright A$. In particular, there exists finite index subgroups $H,K \subset \Gamma_1 \times \dots \times \Gamma_n$ and an isomorphism $\theta : H \rightarrow K$ such that $\alpha(u_g) \in \T u_{\theta(g)}$ for all $g \in H$. By Lemma \ref{permutation lattice}, there exists a permutation $\sigma$ of $\{1, \dots, n\}$ such that $\theta(H \cap \Gamma_i)$ and $\Gamma_{\sigma(i)}$ are commensurable for all $i \in \{1, \dots, n \}$. Let $F_i = \theta(H \cap \Gamma_i) \cap \Gamma_{\sigma(i)}$ which has finite index in $\Gamma_{\sigma(i)}$. Then $L(\Gamma_{\sigma(i)}) \prec_M L(F_i) \subset \alpha(L(\Gamma_i))$. By taking relative commutants, we obtain $\alpha(M_i)' \cap M \prec_M (L(\Gamma_{\sigma(i)})' \cap M_{\sigma(i)})  \ovt (M_{\sigma(i)}' \cap M)$. Since $L(\Gamma_{i})' \cap M_i=L(\mathcal{Z}(\Gamma_i))$ is a finite dimensional abelian algebra, we get $\alpha(M_i)' \cap M \prec_M M_{\sigma(i)}' \cap M$. Finally, by taking relative commutants again, we conclude that $M_{\sigma(i)} \prec_M \alpha(M_i)$ for all $i \in \{1, \dots, n \}$ as we wanted.
\end{proof}

Now, we will prove Theorem \ref{compact minimal rigidity}. We will need the following intertwining lemma.

\begin{lem} \label{minimal intertwining}
Let $\Gamma \curvearrowright N$ be a minimal action of an ICC group $\Gamma$ on a $\II_1$ factor $N$ and let $M = N \rtimes \Gamma$. Let $P \in \mathrm{TF}(M)$ and suppose that $L(\Gamma) \prec_M P$. Then there exists a tensor product decomposition $N=A \ovt B$ with $\Gamma$ acting trivially on $B$ and a unitary $u \in \mathcal{U}(M)$ such that $uPu^*=A \rtimes \Gamma$.
\end{lem}
\begin{proof}
Since $\Gamma$ is ICC and acts minimally on $N$, we have that $L(\Gamma)$ and $L(\Gamma)' \cap M=N^\Gamma$ are $\II_1$ factors. Since by assumption $L(\Gamma) \prec_M P$, the proof of \cite[Proposition 12]{OP03}, shows that we have $L\Gamma \subset P_0$ for some $P_0 \in \mathrm{TF}(M)$ with $P_0 \sim_M P$.
Put $B= P_0' \cap M$ and observe that $B \subset L(\Gamma)' \cap M=N^\Gamma$. Since $B \in \mathrm{TF}(M)$, we can write $N=A \ovt B$ by putting $A=B' \cap N \subset P_0$ and we get $P_0=A \rtimes \Gamma$ as we wanted. Finally, since $P \sim P_0$, we have $P=uP_0^tu^*=u(A^t \rtimes \Gamma)u^*$ for some $u \in \mathcal{U}(M)$ and $t > 0$ where $N=A^t \ovt B^{1/t}$.
\end{proof}

\begin{proof}[Proof of Theorem \ref{compact minimal rigidity}]
The main step is to prove the following fact.
\begin{newclaim} For every tensor product decomposition $M=P\ovt Q$, there exists a unitary $u \in \mathcal{U}(M)$, a direct product decomposition $\Gamma=G \times H$ and a tensor product decomposition $R=A \ovt B$ with $G$ and $H$ acting trivially on $B$ and $A$ respectively, such that $uPu^*=A \rtimes G$ and $uQu^*=B \rtimes H$.
\end{newclaim}
\begin{proof}[Proof of the claim] Consider the double action of $\widehat{\Gamma}=\Gamma_1 \times \Gamma_2$ on $\widehat{R}=R_1 \ovt R_2$ and let $\widehat{M}=M_1 \ovt M_2=\widehat{R} \rtimes \widehat{\Gamma}$. Since the action of $\Gamma$ on $R$ is compact, then so is the action of $\widehat{\Gamma}$ on $\widehat{R}$. Thus, we have $\sigma_P(\widehat{R}) \prec_{\widehat{M}} \widehat{R}$ by \cite[Theorem 4.16]{BIP18}. By Lemma \ref{conjugacy bernoulli}, we can find a unitary $u \in \mathcal{U}(\widehat{M})$ such that $\alpha(\widehat{R})=\widehat{R}$ where $\alpha=\Ad(u) \circ \sigma_P$. Then there exists $\theta \in \Aut(\Gamma_1 \times \Gamma_2)$ such that $\alpha(u_g) \in \widehat{R}u_{\theta(g)}$ for all $g \in \Gamma_1 \times \Gamma_2$ (where $u_g$ is a canonical unitary in $\widehat{R}\rtimes \widehat{\Gamma}$ for $g\in \widehat{\Gamma}$). Since $\Gamma$ is ICC, we can find a direct product decomposition $\Gamma=G \times H$ such that $\theta(\Gamma_2)=G_1 \times H_2$. This implies that 
$$ \alpha(M_2 \ovt R_1)=\alpha(\widehat{R} \rtimes \Gamma_2)=(R_1 \rtimes G_1) \ovt (R_2 \rtimes H_2)=:L.$$ 
Here $G_1$ and $H_2$ have property (T) so that every central sequence of $L$ lies in $R_1^{G_1} \ovt R_2^{H_2}$. But $\alpha(R_1) \in \mathrm{TF}(L)$ is amenable and therefore we get $\alpha(R_1) \prec_{L} R_1^{G_1} \ovt R_2^{H_2}$ (if not, one can construct a central sequence from $\alpha(R_1)$ which is away from $R_1^{G_1} \ovt R_2^{H_2}$). By taking commutants inside $L$, we get $L(G_1) \ovt L(H_2) \prec_L \alpha(M_2)=u(P_1 \ovt Q_2)u^*$. We conclude that $L(G) \prec_M P$ and $L(H) \prec_M Q$. Now, we will show that, up to exchanging $P,Q$ with equivalent ones in $\mathrm{TF}(M)$ and up to unitary conjugating in $M$, we actually have $L(G) \subset P$ and $L(H) \subset Q$.

By applying Lemma \ref{minimal intertwining} to $G$ acting on $N:=R \rtimes H$, we may assume that $L(G) \subset P$, $Q \subset L(G)' \cap N=R^G \rtimes H$ and $Q\in \mathrm{TF}(N)$. We still have $L(H) \prec_M Q$. We claim that we actually have $L(H) \prec_N Q$. Indeed, as an $N$-$N$-bimodule, we have $\rL^2(M)=\bigoplus_{g \in G} \rL^2(N)u_g$. But $L(H)$ and $Q$ are both fixed by $G$. Thus, we have that ${_{L(H)}} \rL^2(M)_Q$ is unitarily equivalent to a mutiple ${_{L(H)}}\rL^2(N)_Q$. This shows that $L(H) \prec_N Q$. Now, by applying Lemma \ref{minimal intertwining} again to $H$ acting on $R$, up to same equivalences as before, we have a tensor product decomposition $R=C \ovt D$ with $H$ acting trivially on $C$ such that $Q=D \rtimes H$ (but a priori, we no longer have $L(G) \subset P$).  Put $$Z:=L(H)' \cap M= R^H \rtimes G$$ and observe that $Z = P\ovt (L(H)' \cap Q)=P\ovt D^H$ (because $L(H)\subset Q \in \mathrm{TF}(M)$). 
Then since $D^H\in \mathrm{TF}(Z)$ is amenable and $G$ has property (T), the same reasoning as above shows that $D^H \prec_Z L(G)' \cap Z=R^\Gamma$. By taking the commutants inside $Z$, we get $L(G) \prec_Z P$. Now, Lemma \ref{minimal intertwining} implies that there exists a unitary $u \in \mathcal{U}(Z)$ such that $L(G) \subset u P u^*$. Since $Z$ commutes with $L(H-$, we still have $L(H) \subset uQu^*$ and therefore we may assume that $L(G) \subset P$ and $L(H) \subset Q$. 

Since $\Gamma$ is ICC, hence also $G$ and $H$, we have $Q \subset L(G)' \cap M=R^G \rtimes H$. Thus $R \rtimes H=A \ovt Q$ for some $A \subset R \rtimes H$. Since $A$ commutes with $L(H) \subset Q$, we have $A \subset R^H$. In particular, $R=A \ovt B$ for some $B \subset Q$. Then $G$ acts trivially on $B$ and $H$ acts trivially on $A$ and $P=A \rtimes G$ and $Q=B \rtimes H$ as we wanted.
\end{proof} We can now prove items $(\rm i)$, $(\rm ii)$ and $(\rm iii)$ of Theorem \ref{compact minimal rigidity}.

$(\rm i)$ Let $P$ be a tensor factor of $M$ which is not of type $\mathrm{I}$. By the claim, $P$ is unitarily conjugate to a factor of the form $A \rtimes G$ where $A$ must be a hyperfinite $\II_1$ factor. Note that $G \curvearrowright A$ is a compact minimal action. By the uniqueness of minimal actions of compact groups on the hyperfinite $\II_1$ factor \cite{MT06}, we know that $\sigma : G \curvearrowright A$ is conjugate to $\sigma \otimes 1 : G \curvearrowright A \ovt R$. Thus $P \cong A \rtimes G \cong (A \rtimes G) \ovt R$ is McDuff.

$(\rm ii)$ Suppose that $M$ is not semi-prime and write $M=P \ovt Q$ where $P$ and $Q$ are nonamenable. Then by the claim, we can assume that $P=A \rtimes G$ and $Q=B \rtimes H$ where $\Gamma=G \times H$ is a nontrivial direct product decomposition and $R=A \ovt B$. If $P$ or $Q$ is not semi-prime, we can repeat this procedure. This must stop at some point because the length of a direct product decomposition of the lattice $\Gamma$ is bounded by its rank.

$(\rm iii)$ Suppose that we have a tensor product decomposition $M=M_1 \ovt \cdots \ovt M_n$ with $M_i=R_i \rtimes \Gamma_i$ as in item $(\rm ii)$ of the theorem. Let $M=P \ovt Q$ be another tensor product decomposition with $P$ and $Q$ nonamenable. By the claim, we may assume that $P=A \rtimes G$ and $Q=B \rtimes H$ with $\Gamma=G \times H$. Since $\Gamma=\Gamma_1 \times \cdots \times \Gamma_n$ is ICC, we have a decomposition $\Gamma_i=G_i \times H_i$ for all $i \leq n$ with $G=G_1 \times \cdots \times G_n$ and $H=H_1 \times \cdots \times H_n$. Let $K=\overline{\Gamma}$ be the closure of $\Gamma$ in $\Aut(R)$. Then we have $K=\overline{\Gamma_1} \times \cdots \times \overline{\Gamma_n}=\overline{G} \times \overline{H}$. Thus $\overline{\Gamma_i}=\overline{G_i} \times \overline{H_i}$. By the uniqueness of the minimal action of $\overline{\Gamma_i}$ on the hyperfinite $\II_1$ factor, we know that the action $\overline{\Gamma_i} \curvearrowright A_i$ is conjugate to a tensor product of a minimal action of $\overline{G_i}$ and a minimal action of $\overline{H_i}$. Since $R_i \rtimes \Gamma_i$ is semi-prime, this forces either $G_i$ or $H_i$ to be amenable, hence finite because it has property (T), hence trivial because $\Gamma$ is ICC. Since this holds for all $i \in \{1, \cdots, n \}$, we conclude that $G=\times_{i \in I} \Gamma_i$ and $H=\times_{j \in J} \Gamma_j$ for some partition $I \sqcup J=\{1, \dots, n\}$. Now, since $R$ is hyperfinite, it has only one nontrivial tensor product decomposition up to conjugacy by an automorphism. Thus, we can find $\theta \in \Aut(R)$ such that $\theta(A)=\ovt_{i \in I} R_i$ and $\theta(B)=\ovt_{j \in J}R_j$. By the uniqueness of the minimal action of $\overline{G}$ and $\overline{H}$ on the hyperfinite $\II_1$ factor, we may assume that $\theta|_A$ is $G$-equivariant and that $\theta|_B$ is $H$-equivariant, hence that $\theta$ is $\Gamma$-equivariant. We conclude that $\theta$ extends to an automorphism of $M$ such that $\theta(P)=\ovt_{i \in I} M_i$ and $\theta(Q)=\ovt_{j \in J} M_j$.
\end{proof}

\section{Full factors without Unique Prime Factorization}

The goal of this section is to provide examples of full factors which do not satsify the UPF property. We first need to recall some definitions and make some general observations about the notion of spectral gap.

We say that a a unitary representation $\pi : G \rightarrow \mathcal{U}(H)$ has \emph{spectral gap} if it has no almost invariant vectors, or equivalently, if there exists a finite set $K \subset G$ such that $\| \frac{1}{|K|}\sum_{g \in K} \pi(g) \| < 1$. Such a set $K$ will be called a \emph{critical set}.

Following \cite[Section 3]{Po06b}, we say that a representation $\pi : G \rightarrow \mathcal{U}(H)$ has \emph{stable spectral gap} if $\pi \otimes \rho$ has spectral gap for any other representation $\rho : G \rightarrow \mathcal{U}(K)$. It is known that $\pi$ has stable spectral gap if and only if the representation $\pi \otimes \bar{\pi}$ of $G$ on $H \otimes \overline{H}$ has spectral gap \cite[Lemma 3.2]{Po06b}.
\begin{prop} \label{stable gap set}
Let $G \curvearrowright I$ be an action of a group $G$ on a set $I$ and let $\pi : G \curvearrowright \ell^2(I)$ be the associated representation. Then $\pi$ has spectral gap if and only if it has stable spectral gap.
\end{prop}
\begin{proof}
This is the idea of \cite{Ch81}. Suppose that $\pi$ has spectral gap and let $\rho : G \curvearrowright H$ be any unitary representation of $G$. Let $\xi_n \in \ell^2(I) \otimes H$ be a sequence of $\pi \otimes \rho$ almost invariant vectors. View each $\xi_n$ as a function $\xi_n : I \ni  i \mapsto \xi_n(i) \in H$. Define a sequence $\eta_n \in \ell^2(I)$ by $\eta_n(i)=\|\xi_n(i)\|$ for all $i \in I$. Then $(\eta_n)_{n \in \N}$ is a sequence of almost invariant vectors for $\pi$. Thus $\| \xi_n\|=\|\eta_n\| \to 0$ when $n \to \infty$. This shows that $\pi \otimes \rho$ has spectral gap. 
\end{proof}

When the representation $\pi$ in Proposition \ref{stable gap set} has (stable) spectral gap, we will simply say that the action $G \curvearrowright I$ has spectral gap. Recall that an ICC group $G$ is \emph{non-inner amenable} \cite{Ef73} if and only if its action on itself by conjugation $\Ad : \Gamma \curvearrowright \Gamma \setminus \{1\}$ has spectral gap.


Similarly, we say that an action of a group $G$ on a $\II_1$ factor $M$ has (stable) spectral gap if the associated Koopman representation of $G$ on $\rL^2(M) \ominus \C$ has (stable) spectral gap. We denote by $\Ad : \mathcal{U}(M) \curvearrowright M$ the canonical action of $\mathcal{U}(M)$ on $M$ by inner automorphisms. By \cite{Co74}, $M$ is full if and only if the action $\Ad$ has spectral gap. The next result shows that it actually has stable spectral gap, like in the group case. The proof is essentially due to A.\ Ioana. We warmly thank him for allowing us to reproduce it here.

\begin{thm} \label{full stable gap}
Let $M$ be a full $\II_1$ factor. Then the adjoint action $\Ad : \mathcal{U}(M) \curvearrowright M$ has stable spectral gap.
\end{thm}

\begin{lem} \label{norm estimate}
Let $M$ be a $\II_1$ factor. Then there exists unitaries $u_1, \dots, u_n \in \mathcal{U}(M)$ such that
$$ \left \| \frac{1}{n} \sum_{k=1}^n u_k \otimes u_k^* \right \| < 1$$
\end{lem}
\begin{proof}
If not, then there exists a net of unit vectors $(\xi_i)_{i \in I}$ in $\rL^2(M) \ovt \rL^2(M)$ such that $(u \otimes u^*) \xi_i - \xi_i \to 0$ for all $u \in \mathcal{U}(M)$, or equivalently $(a \otimes 1 - 1 \otimes a)\xi_i \to 0$ for all $a \in M$. It is easy to see that this implies that $(ab \otimes 1- ba \otimes 1)\xi_i \to 0$ for all $a,b \in M$. But since $M$ is a $\II_1$ factor, we can find $a,b \in M$ such that $ab-ba$ is invertible.
\end{proof}

\begin{proof}[Proof of Theorem \ref{full stable gap}]
Let $H=\rL^2(M) \ominus \C$ and we will show that the action $\Ad\otimes \overline{\Ad}$ on $H\otimes \overline{H}$ has spectral gap. Since $\overline{\Ad}$ is canonically identified with $\Ad$, we consider the action $\Ad\otimes\Ad$ on $H\otimes H$. 
Let $(\xi_i)_{i}$ be a net of vectors in $H \otimes H$ such that $$\lim_i \| (u \otimes u)\xi_i (u \otimes u)^* - \xi_i \|=0$$ for all $u \in \mathcal{U}(M)$. Let $I=\{T \in \B(\rL^2(M \ovt M)) \mid \lim_i \| T \xi_i \|=0 \}$. Observe that $I$ is a norm closed left ideal. Thus we have $a-Ja^*J \in I$ for every $a$ in the $C^*$-algebra $A$ generated by $\{ u \otimes u \mid u \in \mathcal{U}(M) \} \subset M \ovt M$, where $J$ is the canonical conjugation on $\rL^2(M)\otimes \rL^2(M)$. For every $h=h^* \in M$, we have $$h \otimes 1+1 \otimes h = \lim_{t \to 0} \frac{1}{\ri t}(e^{\ri t h} \otimes e^{\ri t h}-1) \in A.$$
By taking linear combinations, we get $x \otimes 1 + 1 \otimes x \in A$ for all $x \in M$. Let $u \in \mathcal{U}(M)$ and $a=u \otimes 1 + 1 \otimes u$. Since $a \in A$, we have $a^*-JaJ \in I$ hence
$$ 1+u \otimes u^* - uJuJ \otimes 1 - u \otimes JuJ=(u \otimes 1)(a^*-JaJ) \in I.$$
By Lemma \ref{norm estimate}, by regarding $\mathcal{U}(M)$ as a discrete group, we can find a probability distribution $\mu_1 \in \ell^{1}(\mathcal{U}(M))^+$ with finite support such that $$ \| \mathbb{E}_{\mu_1}( u \otimes u^*) \|_{\B(H \otimes H)} =1- \varepsilon$$ for some $\varepsilon > 0$, where we used the notation
$$\mathbb{E}_{\mu_1}(u\otimes u^*) :=  \sum_{u\in \mathcal{U}(M)} (u\otimes u^*) \, \mu_1(u).$$
Note that this implies that $$ \| \mathbb{E}_{\nu \ast \mu_1}( u \otimes u^*) \|_{\B(H \otimes H)}  \leq 1- \varepsilon $$ for any other $\nu \in \ell^{1}(\mathcal{U}(M))^+$. 
Since $M$ is full, the representation $\Ad$ has spectral gap \cite{Co75}, so we can find $\mu_2 \in \ell^{1}(\mathcal{U}(M))^+$ such that $$ \| \mathbb{E}_{\mu_2}( u JuJ ) \|_{\B(H) }< 1.$$ 
Since $\mathcal U(M)\ni u \mapsto uJuJ$ is a representation, if we replace $\mu_2$ by $\mu_2^{\ast n}$ for some $n$ large enough, we can actually assume that $$ \| \mathbb{E}_{\mu_2}( u JuJ ) \|_{\B(H) } < \frac{\varepsilon}{2}.$$ Similarly, since $M$ is nonamenable, the representation $\mathcal U(M)\ni u \mapsto u\otimes JuJ$ has spectral gap, so we can find $\mu_3$ such that $$\|  \mathbb{E}_{\mu_3}( u \otimes JuJ ) \|_{\B(H \otimes H) }  < \frac{\varepsilon}{2}.$$ Finally, by letting $\mu=\mu_3 \ast \mu_2 \ast \mu_1$ and $f(u) = u \otimes u^* - uJuJ \otimes 1 - u \otimes JuJ$ for $u\in \mathcal{U}(M)$, we obtain
$$ \left \| \mathbb{E}_{\mu}( f(u) ) \right \|_{\B(H \otimes H) } < (1-\varepsilon) +\frac{\varepsilon}{2} + \frac{\varepsilon}{2} = 1.$$ 
Thus $1 + \mathbb{E}_{\mu}( f(u) ) $ is invertible. But $1 + \mathbb{E}_{\mu}( f(u) ) \in I$  because $1+f(u) \in I$ for all $u \in \mathcal{U}(M)$. Since $I$ is a left ideal, we conclude that $1 \in I$, i.e.\ $\lim_i \| \xi_i \|=0$ as we wanted.
\end{proof}

The next theorem provides a class of full $\II_1$ factors without the Unique Prime Factorization property. In fact, these factors have infinitely many tensor product decompositions up to stable unitary conjugacy. Under some assumptions (see \cite{Is16} for the definition of \emph{strong primeness} and examples), these tensor product decompositions can still be completely classified.

\begin{thm}
Let $M$ be a $\II_1$ factor and $G$ an ICC group. Let $\sigma_0 : G \curvearrowright M$ be an action by inner automorphisms and consider the diagonal action $\sigma=\sigma_0^{\N} : G \curvearrowright M^{\ovt \N}$. Let $N=M^{\ovt \N} \rtimes_{\sigma} G$. Then the following properties are satisfied:
\begin{itemize}
\item [$(\rm i)$] For every finite subset $F \subset \N$, $M^{\ovt F} \subset M^{\ovt \N}$ is a tensor factor of $N$ whose relative commutant is isomorphic to $N$. The tensor factors $M^{\ovt F}$ are pairwise not stably unitarily conjugate.
\item [$(\rm ii)$] If $G$ is non-inner amenable and $\sigma_0$ has stable spectral gap then $N$ is a full factor.
\item [$(\rm iii)$] If moreover $M$ is strongly prime and $G$ is a hyperbolic group, then for every $P \in \mathrm{TF}(N)$, there exists a finite subset $F \subset \N$ such that $P \sim M^{\ovt F}$ or $P^c \sim M^{\ovt F}$. In particular, $N$ does not admit any prime factorization.
\end{itemize}
\end{thm}
\begin{proof}
$(\rm i)$ This follows from the assumption that $\sigma_0$ is inner, which implies that $\sigma$ is also inner on every $M^{\ovt F}$ where $F \subset \N$ is a finite subset.

$(\rm ii)$ By \cite{Ch81}, it is enough to show that $\sigma$ has spectral gap. Let $H=\rL^2(M) \ominus \C$. Let $\pi_0 : G \rightarrow H$ the representation associated to $\sigma_0$. Then the representation associated to $\sigma$ is $\pi=\bigoplus_{F \subset \N} \pi_0^{\otimes F}$ on $\bigoplus_{F \subset \N} H^{\otimes F}$. In particular, $\pi$ is equivalent to $\pi_0 \otimes \pi'$ for some representation $\pi'$. Since $\pi_0$ has stable spectral gap, we conclude that $\pi$ has spectral gap.

$(\rm iii)$ Let $B=M^{\ovt \N}$. By the proof of \cite[Theorem C and Proposition 7.1.(1)]{Is16}, we have $P \prec_N B$ or $P^c \prec_N B$. Assume without loss of generality that $P \prec_N B$. Then $\rL^2(P)$ is contained in ${_P}\rL^2\langle N,B \rangle_P$. Since $B$ is the increasing union of $M^{\ovt F}$ over finite subsets $F \subset \N$, Proposition \ref{convergence weak containment} implies
$$\rL^2 \langle N, B \rangle  \prec  \bigoplus_{F} \rL^2 \langle N, M^{\ovt F} \rangle $$
as $P$-$P$-bimodules.
By Lemma \ref{key rigidity lemma}, we get $P \prec_N M^{\ovt F}$ for some finite subset $F \subset \N$. Since $M$ is strongly prime, we get $P \sim_N M^{\ovt K}$ for some subset $K \subset F$ (see the proof of \cite[Proposition D]{Is16}).
\end{proof}

\begin{example}
Let $G$ be any non inner amenable ICC group. Let $M=L(G)$ and $\sigma_0 : G \curvearrowright M$ the action by inner conjugation. Then the assumptions of $(\rm ii)$ are satisfied thanks to Proposition \ref{stable gap set}. If $G$ is hyperbolic, then $(\rm iii)$ is also satisfied \cite{Is16}.
\end{example}

\begin{example}
Let $M$ be any full $\II_1$ factor. By Theorem \ref{full stable gap}, we can find a critical family of unitaries $u_1,\dots,u_n \in \mathcal{U}(M)$ witnessing the stable spectral gap of $\Ad : \mathcal{U}(M) \curvearrowright M$. Define an action $\sigma_0 : \F_n \curvearrowright M$ of the free group $\F_n=\langle a_1,\dots,a_n \rangle$ by letting $\sigma_0(a_k)=\Ad(u_k)$ for all $k \in \{1,\dots,n\}$. Then $\sigma_0$ has stable spectral gap and $\F_n$ is non-inner amenable. Thus $N=M^{\ovt \N} \rtimes_{\sigma} \F_n$ is a full $\II_1$ factor which satisfies $N \cong N \ovt M$. Note that $\F_n$ is hyperbolic, so if $M$ is strongly prime, then property $(\rm iii)$ is also satisfied.
\end{example}

\begin{cor}
For every full (separable) $\II_1$ factor $M$, there exists a full (separable) $\II_1$ factor $N$ such that $N \cong N \ovt M$.
\end{cor}

Our next result provides an example of a full $\II_1$ factor $M$ such that $M \cong M \ovt M$. In fact, we will use a remarkable construction due to Meier \cite{Me82} of a \emph{finitely generated} group $G$ such that $G \cong G \times G$. Let us recall the construction of $G$. Consider first the following group 
$$ T= \langle a,b,s,t \mid ta^2t^{-1}=a^3, \: sb^2s^{-1}=b^3,\:  t=[b,sbs^{-1}],\:  s=[a,tat^{-1}] \rangle.$$
Note that $\{a,t\}$ and $\{b,s\}$ generate two copies of the Baumslag-Solitar group
$$\mathrm{BS}(2,3)=\langle a,t \mid ta^2t^{-1}=a^3 \rangle.$$
Moreover, $t$ and $[a,tat^{-1}]$ freely generate a free subgroup of rank $2$ inside $\mathrm{BS}(2,3)$. Therefore, we can think of $T$ as an amalgamated free product of two copies of $\mathrm{BS}(2,3)$ with amalgamation over $\F_2$. Next, we consider $T^\N$ the infinite product group (which is uncountable) and we embed $T \subset T^\N$ diagonally. We define an element $h=(a,a^2,a^3,a^4,\dots) \in T^\N$ and finally we let $G$ be the subgroup of $T^\N$ generated by $T$ and $h$. Then one can show that the isomorphism
\begin{gather*}T^\N \rightarrow T^\N \times T^\N\\
(x_n)_{n \in \N} \mapsto (x_{2n})_{n \in \N} \times (x_{2n+1})_{n \in \N}
\end{gather*}
sends $G$ onto $G \times G$. In particular, $G \cong G \times G$.

\begin{prop} \label{meier group}
Meier's group $G$ is non-inner amenable. In particular $M=L(G)$ is a full separable $\II_1$ factor which satisfies $M \cong M \ovt M$.
\end{prop}
\begin{proof}
First, we observe that $T$ is non-inner amenable. Indeed, $T$ is an amalgamated free product with amalgamation over a free group of rank $2$. By \cite[Theorem 1.1]{DTW19}, we know that any conjugacy invariant mean on $T$ must be supported on the amalgam hence trivial because free groups of rank $2$ are non-inner amenable. Thus $T$ is not inner amenable.

Now, consider the action of $T$ by conjugation on $T^{\N} \setminus \{1\}$ diagonally. We claim that this action has spectral gap and this will imply a fortiori that $G$ is non-inner amenable because $T \subset G \subset T^{\N}$. Let $\pi : T \curvearrowright \ell^2(T^{\N} \setminus \{1\})$ be the associated representation. Observe that a sequence $(t_n)_{n \in \N} \in T^{\N}$ is non-trivial if and only if there exists at least one $n \in \N$ such that $t_n \neq 1$. This shows that $\pi$ is contained in a multiple of $\pi_0 \otimes \rho$ where $\pi_0 : T \curvearrowright \ell^2(T \setminus \{1 \})$ and $\rho : T \curvearrowright \ell^2(T^{\N})$. Since $T$ is non-inner amenable, $\pi_0$ has spectral gap. Thus $\pi_0 \otimes \rho$ also has spectral gap by Proposition \ref{stable gap set}. We conclude that $\pi$ has spectral gap.
\end{proof}

\appendix

\bibliographystyle{plain}

\begin{thebibliography}{AHHM18}

\bibitem[AD95]{AD95} {\sc C. Anantharaman-Delaroche}, {\it Amenable correspondences and approximation properties for von Neumann algebras.} Pacific J. of Math. {\bf 171} (1995), 309--341.

\bibitem[AH12]{AH12} {\sc H. Ando, U. Haagerup}, {\it Ultraproducts of von Neumann algebras.} J. Funct. Anal. {\bf 266} (2014), 6842--6913.

\bibitem[AHHM18]{AHHM18} {\sc H. Ando, U. Haagerup, C. Houdayer, A. Marrakchi}, {\it Structure of bicentralizer algebras and inclusions of type $\III$ factors.} Preprint (2018). {\tt arXiv:1804.05706}


\bibitem[BH16]{BH16} {\sc R. Boutonnet, C. Houdayer}, {\it Amenable absorption in amalgamated free product von Neumann algebras.} Kyoto J. Math. {\bf 58} (2018), no. 3, 583--593.

\bibitem[BIP18]{BIP18} {\sc R. Boutonnet, A. Ioana, J. Peterson}, {\it Properly proximal groups and their von Neumann algebras.} Preprint (2018). {\tt arXiv:1809.01881}


\bibitem[BMO19]{BMO19} {\sc J. Bannon, A. Marrakchi, N. Ozawa}, {\it Full factors and co-amenable inclusions.} Preprint (2019). {\tt arXiv:1903.05395}





\bibitem[CKP14]{CKP14} {\sc I. Chifan, Y. Kida, S. Pant}, {\it Primeness results for von Neumann algebras associated with surface braid groups.} Int. Math. Res. Not. {\bf 16} (2016), 4807--4848.

\bibitem[CPS11]{CPS11} {\sc I. Chifan, S. Popa, J. O. Sizemore}, {\it Some OE- and W$^*$-rigidity results for actions by wreath product groups.} J. Funct. Anal. {\bf 263} (2012), no. 11, 3422--3448.

\bibitem[CSU11]{CSU11} {\sc I. Chifan, T. Sinclair, B. Udrea}, {\it On the structural theory of $\rm II_1$ factors of negatively curved groups, $\rm II$. Actions by product groups.} Adv. Math. {\bf 245} (2013), 208--236.

\bibitem[Ch81]{Ch81} {\sc M. Choda}, {\it Inner amenability and fullness.} Proc. Amer. Math. Soc. {\bf 86} (1982), 663--666.

%
\bibitem[Co74]{Co74} {\sc A. Connes}, {\it Almost periodic states and factors of type ${\rm III_1}$.} J. Funct. Anal. {\bf 16} (1974), 415--445.


\bibitem[Co75]{Co75} {\sc A. Connes}, {\it Classification of injective factors. Cases ${\rm II_1}$, ${\rm II_\infty}$, ${\rm III_\lambda}$, $\lambda \neq 1$.} Ann. of Math. {\bf 74} (1976), 73--115.

\bibitem[Co80]{Co80} {\sc A. Connes}, {\it A factor of type $\II_1$ with countable fundamental group.} J. Operator Theory {\bf 4} (1980), 151--153.

\bibitem[Cor08]{Cor08} {\sc Y. Cornulier}, {\it Dense subgroups with Property (T) in Lie groups.} Comment. Math. Helv. {\bf 83}, 55--65, 2008.




\bibitem[DHI17]{DHI17} {\sc D. Drimbe , D. Hoff, A. Ioana}, {\it Prime $\II_1$ factors arising from irreducible lattices in products of rank one simple Lie groups} To appear in J. Reine. Angew. Math. {\tt arXiv:1611.02209}

\bibitem[DTW19]{DTW19} {\sc B. Duchesne , R. Tucker-Drob, P. Wesolek}, {\it CAT(0) cube complexes and inner amenability.} Preprint (2019). {\tt arXiv:1903.01596}.



\bibitem[Ef73]{Ef73} {\sc E. G. Effros}, {\it Property Γ and inner amenability.} Proc. Amer. Math. Soc. {\bf 47} (1975), 483--486.

\bibitem[Fu09]{Fu09} {\sc A. Furman}, {\it A survey of measured group theory.} Geometry, rigidity, and group actions, 296--374, Chicago Lectures in Math., Univ. Chicago Press, Chicago, IL, 2011.

\bibitem[GP03]{GP03} {\sc D. Gaboriau, S. Popa}, {\it An Uncountable Family of Non Orbit Equivalent Actions of $\F_n$.} J. Amer. Math. Soc. {\bf 18} (2005), 547--559.



\bibitem[Ho15]{Ho15} {\sc D.J. Hoff}, {\it Von Neumann algebras of equivalence relations with nontrivial one-cohomology.} J. Funct. Anal. {\bf 270} (2016), 1501--1536.



\bibitem[HI15]{HI15} {\sc C. Houdayer, Y. Isono}, {\it Unique prime factorization and bicentralizer problem for a class of type ${\rm III}$ factors.} Adv. Math. {\bf 305} (2017), 402--455.



\bibitem[Hj02]{Hj02} {\sc G. Hjorth}, {\it A converse to Dye' s Theorem.} Trans Amer. Math. Soc. {\bf 357} (2004), 3083--3103.

\bibitem[HMV16]{HMV16} {\sc C. Houdayer, A. Marrakchi, P. Verraedt}, {\it Fullness and Connes' $\tau$ invariant of type ${\rm III}$ tensor product factors.} J. Math. Pures Appl. {\bf 121} (2019), 113--134.



%

\bibitem[Io06]{Io06} {\sc A. Ioana}, {\it Rigidity results for wreath product $\rm II_1$ factors.} J. Funct. Anal. {\bf 252} (2007), no. 2, 763--791.

\bibitem[Io08]{Io08} {\sc A. Ioana}, {\it Cocycle Superrigidity for Profinite Actions of Property (T) Groups.} Duke Math. J. {\bf 157} (2011), no. 2, 337--367

\bibitem[Is14]{Is14} {\sc Y. Isono}, {\it Some prime factorization results for free quantum group factors.} 
J. Reine Angew. Math. {\bf 722} (2017), 215--250.

\bibitem[Is16]{Is16} {\sc Y. Isono}, {\it On fundamental groups of tensor product $\II_1$ factors.} To appear in J. Inst. Math. Jussieu. {\tt arXiv:1608.06426.} 



\bibitem[Is19]{Is19} {\sc Y. Isono}, {\it Unitary conjugacy for type $\rm III$ subfactors and W$^*$-superrigidity.} Preprint 2019. ArXiv:1902.01049.






\bibitem[Me82]{Me82} {\sc D. Meier}, {\it Non-hopfian groups.} J. London Math. Soc.(2) {\bf 26} (1982), 265--270.


\bibitem[Ma16]{Ma16} {\sc A. Marrakchi}, {\it Solidity of type $\mathrm{III}$ Bernoulli crossed products.} Comm. Math. Phys. {\bf 350}, (2017), 897--916.



\bibitem[Ma18]{Ma18} {\sc A. Marrakchi}, {\it Full factors, bicentralizer flow and approximately inner automorphisms.} Preprint (2018). {\tt arXiv:1811.10253}

\bibitem[MT06]{MT06} {\sc T. Masuda, R. Tomatsu}, {\it Classification of minimal actions of a compact Kac algebra with amenable dual.} Comm. Math. Phys. {\bf 274} (2007), no. 2, 487--551.


\bibitem[Oc85]{Oc85} {\sc A. Ocneanu}, {\it Actions of discrete amenable groups on von Neumann algebras.} Lecture Notes in Mathematics, {\bf 1138}. Springer-Verlag, Berlin, 1985. iv+115 pp.

\bibitem[Oz02]{Oz02} {\sc N. Ozawa}, {\it There is no separable universal $\II_1$ factor.} Proc. Amer. Math. Soc. {\bf 132} (2004), 487--490.

\bibitem[Oz03]{Oz03} {\sc N. Ozawa}, {\it Solid von Neumann algebras.} Acta Math.\ {\bf 192} (2004), 111--117.


\bibitem[OP03]{OP03} {\sc N. Ozawa, S. Popa} {\it Some prime factorization results for type $\II_1$ factors.} Invent. Math. {\bf 156} (2004), 223--234.



\bibitem[Pe06]{Pe06} {\sc J. Peterson}, {\it $\rL^2$-rigidity in von Neumann algebras.} Invent. Math. {\bf 175} (2009), 417--433.



\bibitem[Po86]{Po86}{\sc S. Popa}, {\it Correspondences}. INCREST preprint (1986).


\bibitem[Po01]{Po01} {\sc S. Popa}, {\it On a class of type ${\rm II_1}$ factors with Betti numbers invariants.} Ann. of Math. {\bf 163} (2006), 809--899.

\bibitem[Po03]{Po03} {\sc S. Popa}, {\it Strong rigidity of ${\rm II_1}$ factors arising from malleable actions of w-rigid groups $\rm I$.} Invent. Math. {\bf 165} (2006), 369--408.

\bibitem[Po06a]{Po06a} {\sc S. Popa}, {\it On Ozawa's property for free group factors.} Int. Math. Res. Not. IMRN (2007), no. 11, Art. ID rnm036, 10 pp.

\bibitem[Po06b]{Po06b} {\sc S. Popa}, {\it On the superrigidity of malleable actions with spectral gap.} J. Amer. Math. Soc. {\bf 21} (2008), 981--1000

\bibitem[Po06c]{Po06c} {\sc S. Popa}, {\it Deformation and rigidity for group actions and von Neumann algebras.} Proceedings of the
International Congress of Mathematicians” (2006), Volume I, EMS Publishing House, Zurich, 2006/2007, 445--479.



\bibitem[PV10]{PV10} {\sc S. Popa, S. Vaes}, {\it Actions of $\F_\infty$ whose $\II_1$ factors and orbit equivalence relations have prescribed fundamental group.} J. Amer. Math. Society {\bf 23} (2010), 383--403.


\bibitem[PV11]{PV11} {\sc S. Popa, S. Vaes}, {\it On the fundamental group of $\II_1$ factors and equivalence relations arising from group actions.} Quanta of Maths, Proc. Conf. in honor of A.Connes' 60th birthday, Clay Math. Inst. Proc. 11 (2011), 519--541.

\bibitem[Ra91]{Ra91} {\sc F. Radulescu}, {\it The fundamental group of the von Neumann algebra of a free group with infinitely many generators is $\R_+^*$}. J. Amer. Math. Soc. {\bf 5} (1992), 517--532.


\bibitem[Sa09]{Sa09} {\sc H. Sako}, {\it Measure equivalence rigidity and bi-exactness of groups.} J. Funct. Anal. {\bf 257} (10) (2009) 3167--3202.


\bibitem[Su80]{Su80} {\sc C. E. Sutherland}, {\it Cohomology and Extensions of von Neumann Algebras {\rm I}.} Publ. RIMS, Kyoto Univ.
{\bf16} (1980), 105--133.

\bibitem[SW11]{SW11} {\sc J.O. Sizemore, A. Winchester}, {\it A unique prime decomposition result for wreath product factors.} Pacific J.\ Math.\ {\bf 265} (2013), 221--232.






%
%




\bibitem[VV14]{VV14} {\sc S. Vaes, P. Verraedt}, {\it Classification of type ${\rm III}$ Bernoulli crossed products.} Adv. Math. {\bf 281} (2015), 296--332.


\bibitem[Vo89]{Vo89} {\sc D.V. Voiculescu}, {\it Circular and semicircular systems and free product factors}. In Operator algebras, unitary representations, enveloping algebras, and invariant theory (Paris, 1989), Progr. Math. {\bf 92}, Birkh$\rm \ddot{a}$user, Boston, 1990, p. 45--60.


\end{thebibliography}

\end{document}